\documentclass[12pt,draftcls,onecolumn]{IEEEtran}  % Comment this line out
\IEEEoverridecommandlockouts                              % This command is only
\usepackage{epsfig} % for postscript graphics files
\usepackage{amsmath} % assumes amsmath package installed
\usepackage{amssymb}  % assumes amsmath package installed
\usepackage{subfigure}
\usepackage{color}
\usepackage{dsfont}
\newtheorem{lemma}{Lemma}

\newtheorem{theorem}{Theorem}

\numberwithin{equation}{section}
\numberwithin{figure}{section}
\numberwithin{lemma}{section}
\numberwithin{theorem}{section}
\numberwithin{definition}{section}

\title{\LARGE \bf Supplementary Material to: Passive Controller Realization of a Biquadratic Impedance with Double Poles and Zeros
as a Seven-Element Series-Parallel Network for Effective Mechanical Control}
\date{}

\author{Kai~Wang$^{1}$, ~~Michael~Z.~Q.~Chen$^{2,\ast}$, ~\IEEEmembership{Senior Member,~IEEE,}~~
Chanying Li$^{3,4}$, ~\IEEEmembership{Member,~IEEE,}~~ and
Guanrong Chen$^{5}$,~\IEEEmembership{Fellow,~IEEE}
\thanks{$^{1}$Key Laboratory of Advanced Process Control for Light Industry (Ministry of Education), School of Internet of Things Engineering, Jiangnan University, Wuxi, Jiangsu 214122, P. R. China.}
\thanks{$^{2}$School of Automation, Nanjing University of Science and Technology, Nanjing, Jiangsu 210094, P. R. China.}
\thanks{$^{3}$Key Laboratory of Systems and Control, Academy of Mathematics and Systems Science, Chinese Academy of Sciences, Beijing 100190, China.}
\thanks{$^{4}$School of Mathematical Sciences, University of Chinese Academy of Sciences, Beijing 100049, China.}
\thanks{$^{5}$Department of Electronic Engineering, City University of Hong Kong, Tat Chee Avenue, Kowloon, Hong Kong.}
\thanks{This work is supported by the National Natural Science Foundation of China under grants 61703184 and 61374053.}
\thanks{Correspondence: MZQ Chen, mzqchen@outlook.com.}
}

\begin{document}

\maketitle
\thispagestyle{empty}
%\pagestyle{empty}

% As a general rule, do not put math, special symbols or citations
% in the abstract or keywords.

\begin{abstract}
This report presents some supplementary material to  the paper entitled   ``Passive controller realization of a biquadratic impedance with double poles and zeros
as a seven-element Series-parallel network for effective mechanical control'' \cite{WCLC18}.

\medskip


\end{abstract}

% For peer review papers, you can put extra information on the cover
% page as needed:
% \ifCLASSOPTIONpeerreview
% \begin{center} \bfseries EDICS Category: 3-BBND \end{center}
% \fi
%
% For peerreview papers, this IEEEtran command inserts a page break and
% creates the second title. It will be ignored for other modes.
\IEEEpeerreviewmaketitle

\section{Introduction}
This report presents some supplementary material to the paper entitled   ``Passive controller realization of a biquadratic impedance with double poles and zeros as a seven-element Series-parallel network for effective mechanical control'' \cite{WCLC18}, which are omitted from the paper for brevity.

\section{Preliminaries of Passive Network Synthesis}

A two-terminal electrical network is defined to be \textit{passive} if $\int_{-\infty}^{T} i(t)v(t) dt \geq 0$ for all $T$ and for all admissible current $i(t)$ and voltage $v(t)$ \cite{AV73}.
A real-rational function $F(s)$ is \textit{positive real} if $F(s)$ is analytic and $\mathfrak{R}(F(s)) \geq 0$ for any $\mathfrak{R}(s) > 0$ \cite{Gui57}. An \textit{impedance} (resp. \textit{admittance}) is defined to be $Z(s) = V(s)/I(s)$ (resp. $Y(s) = I(s)/V(s)$), where $V(s)$ and $I(s)$ are voltages and currents of the port of a two-terminal network, where  the network is said to \textit{realize} (or is a realization of) its impedance (resp. admittance).
A two-terminal electrical network is passive if and only if its impedance (resp. admittance) is positive real,
and any positive-real impedance (resp. admittance) is realizable as a two-terminal passive $RLC$ network \cite{Gui57}.
Any positive-real impedance (resp. admittance) is realizable as a two-terminal passive $RLC$ network.
A \textit{reactive element} is an inductor or capacitor, and a resistor is also called as a   \textit{resistive element}.

\section{Definitions of the network duality and the frequency inverse}

Any two-terminal passive $RLC$ network $N$ can be regarded as a \textit{one-terminal-pair labeled graph} $\mathcal{N}$ with two distinguished \textit{terminal vertices} (see \cite[pg.~14]{SR61}), in which the labels designate passive circuit elements regardless of element values, namely resistors, capacitors, and inductors, which are labeled as $R_i$, $C_i$, and $L_i$, respectively.

Two natural maps acting on the labeled graph are defined as follows:
\begin{enumerate}
  \item $\text{GDu} :=$ Graph duality, which takes the one-terminal-pair graph into its dual (see \cite[Definition~3-12]{SR61}) while preserving the labeling.
  \item $\text{Inv}  :=$ Inversion, which preserves the graph but interchanges the reactive elements, that is, capacitors to inductors and inductors to capacitors with their labels $C_i$ to $L_i$ and $L_i$ to $C_i$.
\end{enumerate}
Consequently, one defines
\begin{equation*}
\text{Dual} :=  \text{network duality of one-terminal-pair labeled graph} := \text{GDu} \circ \text{Inv} = \text{Inv} \circ \text{GDu}.
\end{equation*}
%Then, the concept of the \textit{network duet} is introduced as follows.
%\begin{definition}
%A network duet is a commutative diagram of two one-terminal-pair  labeled graphs $\mathcal{N}$ and $\text{Dual}(\mathcal{N})$. Specially, if $\mathcal{N} = \text{Dual}(\mathcal{N})$, then the duet contains one one-terminal-pair labeled graph.
%\end{definition}
%\begin{definition}
%A network quartet is a commutative diagram of four one-terminal-pair  labeled graphs $\mathcal{N}$, $\text{NetDu}(\mathcal{N})$, $\text{Inv}(\mathcal{N})$, and $\text{GrDu}(\mathcal{N})$. Specially, if $\mathcal{N} = \text{NetDu}(\mathcal{N})$
%or (resp. and) $\mathcal{N} = \text{Inv}(\mathcal{N})$, then the quartet contains two one-terminal-pair labeled graphs (resp.  one one-terminal-pair labeled graph).
%\end{definition}

Consider a network $N$ whose one-terminal-pair labeled graph is $\mathcal{N}$.
Denote $\text{Inv}(N)$ as the network whose one-terminal-pair labeled graph is $\text{Inv}(\mathcal{N})$,  resistors are of the same values as those of $N$, and  inductors (resp. capacitors) are replaced by capacitors (resp. inductors) with reciprocal values, which is called the \textit{frequency inverse network} of $N$.
Denote $\text{GDu}(N)$  as the network whose one-terminal-pair labeled graph is $\text{GDu}(\mathcal{N})$ and elements are of the reciprocal values to those of $N$, which is called the \textit{frequency inverse dual network} of $N$.
Denote $\text{Dual}(N)$ as the network whose one-terminal-pair labeled graph is $\text{Dual}(\mathcal{N})$,   resistors  are of reciprocal values to those of $N$, and   inductors (resp. capacitors) are replaced by capacitors (resp. inductors) with same values, which is called the \textit{dual network} of $N$.

It can be proved that $Z(s)$ (resp. $Y(s)$) is realizable as the impedance (resp. admittance)
of a network $N$ whose one-terminal-pair labeled graph is $\mathcal{N}$,
if and only if $Z(s^{-1})$ (resp. $Y(s^{-1})$) is realizable as the impedance (resp. admittance) of $\text{Inv}(N)$ whose one-terminal-pair labeled graph is $\text{Inv}(\mathcal{N})$, if and only if
$Z(s^{-1})$ (resp. $Y(s^{-1})$) is realizable as the admittance (resp. impedance) of $\text{GDu}(N)$ whose one-terminal-pair labeled graph is $\text{GDu}(\mathcal{N})$, and if and only if
it is realizable as
the admittance (resp. impedance) of $\text{Dual}(N)$ whose one-terminal-pair labeled graph is $\text{Dual}(\mathcal{N})$.
Therefore, if a necessary and sufficient condition  is derived for
$H(s) =
\sum_{k=0}^m a_k s^k/\sum_{k=0}^m b_k s^k$ to be realizable as the impedance (resp. admittance) of a two-terminal network  whose one-terminal-pair labeled graph is $\mathcal{N}$, then the corresponding condition for $\text{Inv}(\mathcal{N})$ can be obtained from that for $\mathcal{N}$ through conversion $a_k \leftrightarrow a_{m-k}$ and $b_k \leftrightarrow b_{m-k}$ for $k = 0,1,...,\lfloor m/2 \rfloor$ (the \textit{principle of frequency inversion}).
The corresponding condition for  $\text{GDu}(\mathcal{N})$ can be obtained from that for $\mathcal{N}$ through conversion $a_k \leftrightarrow b_{m-k}$ for $k = 0,1,...,m$ (the \textit{principle of frequency-inverse duality}).
Furthermore, the  corresponding condition for $\text{Dual}(\mathcal{N})$ can be obtained from that for $\mathcal{N}$ through conversion $a_k \leftrightarrow b_k$ for $k = 0, 1,..., m$ (the \textit{principle of duality}).

Specifically, based on the principle of frequency inversion,
a necessary and sufficient condition for $Z(s)$ in the form of (1) with $k$, $z$, $p$ $> 0$ to be realizable as the impedance  of a two-terminal network whose one-terminal-pair labeled graph is $\text{Inv}(\mathcal{N})$ can be obtained from that for $\mathcal{N}$ through $z \leftrightarrow z^{-1}$ and $p \leftrightarrow p^{-1}$.
Based on the principle of frequency-inversion duality, a necessary and sufficient condition for $Z(s)$ in the form of (1) with $k$, $z$, $p$ $> 0$ to be realizable as the impedance  of a two-terminal network whose one-terminal-pair labeled graph is $\text{GDu}(\mathcal{N})$ can be obtained from that for $\mathcal{N}$  through $p \leftrightarrow z^{-1}$ and $z \leftrightarrow p^{-1}$.
Based on the principle of duality, a necessary and sufficient condition for $Z(s)$ in the form of (1) with $k$, $z$, $p$ $> 0$ to be realizable as the impedance of a two-terminal network whose one-terminal-pair labeled graph is $\text{Dual}(\mathcal{N})$ can be obtained from that for $\mathcal{N}$  through $p \leftrightarrow z$.

%the corresponding condition as $N^i$ whose one-terminal-pair labeled graph is $\text{Inv}(\mathcal{N})$ can
%be obtained through $a_i \leftrightarrow a_{m-i}$ and
%$b_i \leftrightarrow b_{m-i}$ for $0 \leq i \leq \lfloor m/2 \rfloor$, and the corresponding condition
%as $N^{id}$ whose one-terminal-pair labeled graph is $\text{GrDu}(\mathcal{N})$ can be obtained
%through $a_i \leftrightarrow b_{m-i}$ and $b_i \leftrightarrow a_{m-i}$ for $0 \leq i \leq \lfloor m/2 \rfloor$.

\section{Realizability as Series-Parallel Networks with No More Than Five Elements}

\begin{theorem}  \label{theorem: fewer than four elements}
{A biquadratic impedance $Z(s)$ in the form of (1), that is,
\begin{equation*}
Z(s) = k \frac{(s + z)^2}{(s + p)^2},
\end{equation*}
where $k$, $z$, $p$ $> 0$ and $p \neq z$, cannot be realized with fewer than four elements.}
\end{theorem}
\begin{proof}
Let $A = kx$, $B = 2kzx$, $C = k z^2x$, $D = x$, $E = 2px$, and $F = p^2x$ for  $x > 0$. This theorem can be proved from the realizability conditions of a general biquadratic impedance in the form of (2), that is,
\begin{equation}  \label{eq: general biquadratic impedances}
Z(s) = \frac{A s^2 + B s + C}{D s^2 + E s + F},
\end{equation}
where $A$, $B$, $C$, $D$, $E$, $F$ $> 0$, with at most three elements in \cite{WCH14}.
\end{proof}

\begin{theorem}  \label{theorem: four elements}
{A biquadratic impedance $Z(s)$ in the form of (1), where $k$, $z$, $p$ $> 0$ and $p \neq z$, is realizable as a four-element series-parallel network if and only if $p = z/3$ or $p = 3z$.}
\end{theorem}
\begin{proof}
Let $A = kx$, $B = 2kzx$, $C = k z^2x$, $D = x$, $E = 2px$, and $F = p^2x$ for  $x > 0$. This condition can be derived from the realizability conditions of a general biquadratic impedance in the form of (2) where $A$, $B$, $C$, $D$, $E$, $F$ $> 0$ as a four-element network in \cite{WCH14}, where it is obvious that any four-element network must be series-parallel.
\end{proof}

\begin{theorem}  \label{theorem: five elements}
{A biquadratic impedance $Z(s)$ in the form of (1), where $k$, $z$, $p$ $> 0$ and $p \neq z$, is realizable as a five-element series-parallel network if and only if $p/z \in (1/3, 3)$,   $p = (2 + \sqrt{2})z$, or $p = z/(2 + \sqrt{2})$.}
\end{theorem}
\begin{proof}
Let $A = kx$, $B = 2kzx$, $C = k z^2x$, $D = x$, $E = 2px$, and $F = p^2x$ for  $x > 0$. This condition can be derived from the realizability conditions of a general biquadratic impedance in the form of (2) where $A$, $B$, $C$, $D$, $E$, $F$ $> 0$ as a five-element series-parallel network in \cite{JS11}.
\end{proof}

\section{Some Basic Lemmas}

\begin{lemma}
\label{lemma: biquadratic impedances Z2 three elements}
{Consider a biquadratic impedance $F(s)$ in the form of
\begin{equation} \label{eq: biquadratic impedances Z2}
F(s) =  \frac{\alpha s^2 + \beta s + \gamma}{(s + p)^2},
\end{equation}
where $\alpha$, $\beta$, $\gamma$ $\geq 0$, and $p > 0$.
Then, $F(s)$ is realizable as a three-element series-parallel network if and only if at least one of the following conditions holds: 1. $\alpha \gamma = 0$; 2. $\beta = 0$ and $\alpha p^2 - \gamma  = 0$; 3. $\gamma = 0$ and $\alpha p - 2 \beta = 0$; 4. $\alpha = 0$ and $2 \beta p - \gamma = 0$; 5. $\alpha p^2 - \beta p + \gamma = 0$.
}
\end{lemma}
\begin{proof}
Since it is obvious that $F(s)$ in the form of \eqref{eq: biquadratic impedances Z2} is not a \textit{reactance function}  \cite[Definition~3.1]{Bah84}, there is at least one resistor, which means that there are at most two reactive elements.

When the number of reactive elements is at most one, the  \textit{degree} \cite[Section~3.6]{AV73} of $F(s)$ cannot exceed one by \cite[Theorem~4.4.3]{AV73}.
Therefore, there must exist at least one common factor between $(\alpha s^2 + \beta s + \gamma)$ and $(s + p)^2$, which holds if and only if
Condition~5 is satisfied.

When the number of reactive elements is two, there is one resistor. Based on the method of enumeration, one can obtain
Conditions~1--4.
\end{proof}

\begin{lemma}
\label{lemma: biquadratic impedances Z2 Four elements}
{Consider a biquadratic impedance $F(s)$ in the form of
\eqref{eq: biquadratic impedances Z2},
where $\alpha$, $\beta$, $\gamma$ $\geq 0$, $p > 0$, assuming that the condition of Lemma~\ref{lemma: biquadratic impedances Z2 three elements} does not hold. Then, $F(s)$ is realizable as a four-element series-parallel network if and only if at least one of the six conditions holds:
1. $\alpha = 0$ and $\gamma < 2 \beta p$; 2. $\gamma = 0$ and $\alpha p < 2 \beta$; 3. $\alpha$, $\beta$, $\gamma$ $> 0$ and $\alpha p^2 - \gamma = 0$; 4. $\alpha$, $\beta$, $\gamma$ $> 0$, $\alpha p^2 < \gamma$, and
  $3 \alpha p^2 + \gamma - 2\beta p = 0$ or $\beta^2 p^2 + \gamma^2 - \alpha  \gamma p^2 - 2 \beta \gamma p = 0$ holds; 5. $\alpha$, $\beta$, $\gamma$ $> 0$, $\alpha p^2 > \gamma$, and  $\alpha p^2 + 3 \gamma - 2\beta p = 0$
  or $\alpha^2 p^2 + \beta^2   - 2 \alpha \beta p - \alpha \gamma  = 0$ holds; 6. $\alpha$, $\beta$, $\gamma$ $> 0$ and $\alpha^2 p^4 - 2\alpha \beta p^3 + 6 \alpha \gamma p^2   - 2 \beta \gamma p + \gamma^2 = 0$.
Moreover, if one of the above six conditions holds, then $F(s)$ is realizable as a two-reactive four-element series-parallel network.
}
\end{lemma}
\begin{proof}
Let $A = \alpha x$, $B = \beta x$, $C = \gamma x$, $D = x$, $E = 2px$, and $F = p^2x$ for any $x > 0$.
A necessary and sufficient condition for $Z(s)$ in the form of (2) with $A$, $B$, $C$, $D$, $E$, $F$ $\geq 0$ to be positive real is $(\sqrt{AF} - \sqrt{CD})^2 \leq BE$ \cite{CS09(2),Fos62}.
Since any positive-real biquadratic impedance with zero coeffients is realizable as a two-reactive four-element series-parallel network \cite[Lemma~8]{JS11}, one obtains Conditions~1 and 2 together with Lemma~\ref{lemma: biquadratic impedances Z2 three elements}.
Now, it remains to considering the case of $\alpha$, $\beta$, $\gamma$ $> 0$. By \cite[Theorem~5]{WCH14}, one  obtains Conditions~3--6.

By the covering configurations in \cite[Figs.~4--6]{WCH14}, the number of reactive elements for realizations is two.
\end{proof}

\begin{lemma}
\label{lemma: biquadratic impedances Z2 Five elements Two reactive}
{Consider a biquadratic impedance $F(s)$ in the form of
\eqref{eq: biquadratic impedances Z2},
where $\alpha$, $\beta$, $\gamma$, $p > 0$, and neither the condition of Lemmas~\ref{lemma: biquadratic impedances Z2 three elements} nor the condition of Lemma~\ref{lemma: biquadratic impedances Z2 Four elements} holds. Then, $F(s)$ is realizable as a two-reactive five-element  series-parallel network if and only if at least one of the following conditions holds: 1. $\alpha p^2 > \gamma$ and $\alpha p^2 + 3\gamma - 2\beta p  < 0$; 2. $\alpha p^2 > \gamma$ and $\alpha^2 p^2 + \beta^2 - 2\alpha \beta p  - \alpha \gamma < 0$; 3. $\alpha p^2 < \gamma$ and $3 \alpha p^2 + \gamma - 2 \beta p  < 0$; 4. $\alpha p^2 < \gamma$ and $\beta^2 p^2 + \gamma^2 - \alpha  \gamma p^2 - 2 \beta \gamma p < 0$.
}
\end{lemma}
\begin{proof}
A necessary and sufficient condition for $Z(s)$ in the form of (2) where $A$, $B$, $C$, $D$, $E$, $F$ $> 0$ to be realizable as a two-reactive five-element series-parallel network is presented  in \cite[Theorem~1]{JS11}.
Letting $A = \alpha x$, $B = \beta x$, $C = \gamma x$, $D = x$, $E = 2px$, and $F = p^2x$ for any $x > 0$, this lemma can be proved together with the assumption that neither the condition of Lemma~\ref{lemma: biquadratic impedances Z2 three elements} nor the condition of Lemma~ does not hold.
\end{proof}

\section{Supplementary Lemmas of Three-Reactive Seven-Element Series-Parallel Realizations  for the Proof of Lemma~2}

\begin{lemma}
\label{lemma: Three-Reactive-Structure}
{Consider a biquadratic impedance $Z(s)$ in the form of \eqref{eq: general biquadratic impedances}
with $A$, $B$, $C$, $D$, $E$, $F$ $> 0$ that cannot be realized as a series-parallel network containing fewer than seven elements. If $Z(s)$ is realizable as a three-reactive seven-element series-parallel network as shown in Fig.~2(a),
where $N_1$ is a three-element series-parallel network
and $N_2$ is a four-element series-parallel network, then $Z(s)$ is also realizable as shown  in Fig.~\ref{fig: Equivalent structure}, where $N_b$ is a two-reactive five-element series-parallel network.  }
\end{lemma}
\begin{figure}[thpb]
      \centering
      \subfigure[]{
      \includegraphics[scale=1.15]{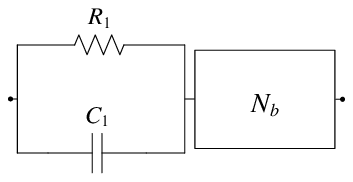}
      \label{fig: Equivalent structure a}}
      \subfigure[]{
      \includegraphics[scale=1.15]{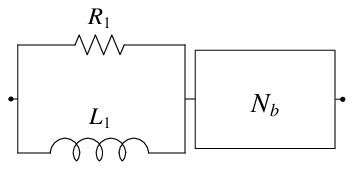}
      \label{fig: Equivalent structure b}}
      \caption{A class of three-reactive seven-element series-parallel networks, where $N_b$ is a two-reactive five-element series-parallel network (Fig.~7).}
      \label{fig: Equivalent structure}
\end{figure}
\begin{proof}
By \cite[Lemma~2]{WCH14}, $Z(s)$ cannot be realized as the series connection of two networks, one of which only contains  reactive elements.
Therefore, $N_1$ must contain one or two reactive elements.
If $N_1$ contains two reactive elements, then $N_2$  contains one reactive element. Therefore,    the  \textit{degree} \cite[Section~3.6]{AV73} of $Z_2(s)$ cannot exceed one by \cite[Theorem~4.4.3]{AV73},  where $Z_2(s)$ denotes the impedance of $N_2$.
By the discussion in \cite{WCH14}, $Z_2(s)$ is realizable as a one-reactive three-element series-parallel network, which implies that $Z(s)$ is realizable with a series-parallel network containing fewer than seven elements. This  contradicts the assumption.
If $N_1$ contains one reactive element, then the  \textit{degree} \cite[Section~3.6]{AV73} of $Z_1(s)$ cannot exceed one by \cite[Theorem~4.4.3]{AV73}, where $Z_1(s)$ denotes the impedance of $N_1$. By the discussion in \cite{WCH14}, $Z_1(s)$ is realizable as a one-reactive three-element series-parallel network, which is equivalent to one of configurations  in Fig.~\ref{fig: configurations Lemma1}. Regarding the series connection of $R_2$ and $N_2$ as $N_b$, this lemma is proved.
\end{proof}

\begin{figure}[thpb]
      \centering
      \subfigure[]{
      \includegraphics[scale=1.15]{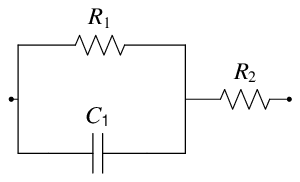}
      \label{fig: configurations Lemma1 a}}
      \subfigure[]{
      \includegraphics[scale=1.15]{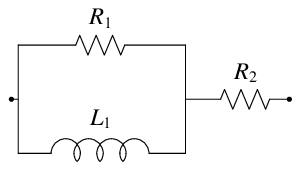}
      \label{fig: configurations Lemma1 b}}
      \caption{Configurations for $N_1$ mentioned in the proof of
      Lemma~\ref{lemma: Three-Reactive-Structure}.}
      \label{fig: configurations Lemma1}
\end{figure}

\section{Complete Proof of Lemma~2}

Assume that $Z(s)$ is realizable as in Fig.~\ref{fig: Equivalent structure}(a), where $N_b$ is a two-reactive five-element series-parallel network.
Let $Z(s) = Z_a(s) + Z_b(s)$, where $Z_a(s)$ is the impedance of the parallel connection of $R_1$ and $C_1$ and $Z_b(s)$ is the impedance of $N_b$. It is clear that $Z_a(s)$ can be written in the form of
\begin{equation*}
Z_a(s) = \frac{m}{s+p},
\end{equation*}
where $m > 0$.
Since the condition  of Theorem~1 does not hold, $N_b$ cannot be equivalent to a network containing fewer than five elements. Therefore, $Z_b(s)$ can be expressed in the form of
\begin{equation*}
Z_b(s) = \frac{\alpha s^2 + \beta s + \gamma}{(s+p)^2},
\end{equation*}
where $\alpha$, $\beta$, $\gamma$ $> 0$, and the condition of
Lemma~\ref{lemma: biquadratic impedances Z2 Five elements Two reactive}
holds. Since
\begin{equation*}
Z_a(s) + Z_b(s) = \frac{\alpha s^2 + (\beta + m) s + (\gamma + mp)}{(s+p)^2},
\end{equation*}
it follows that
$\alpha s^2 + (\beta + m) s + (\gamma + mp) = k(s+z)^2$.
Therefore,
\begin{equation}  \label{eq: N1 two N2 five 01}
\beta = 2 \alpha z - m, ~~~~~ \gamma = \alpha z^2 - mp.
\end{equation}
Since $\beta$, $\gamma$ $> 0$, it follows from \eqref{eq: N1 two N2 five 01} that
\begin{equation}  \label{eq: N1 two N2 five 02}
\alpha > \max\left\{ \frac{m}{2z}, \frac{mp}{z^2} \right\}.
\end{equation}

If $Z_b(s)$ satisfies Condition~1 of Lemma~\ref{lemma: biquadratic impedances Z2 Five elements Two reactive}, then $\alpha p^2 > \gamma$ and $\alpha p^2 + 3\gamma - 2\beta p  < 0$. Together with \eqref{eq: N1 two N2 five 01}, one obtains
\begin{align}
(p-z)(p+z) \alpha  + mp &> 0,  \label{eq: N1 two N2 five 03}  \\
(p-z)(p-3z) \alpha  - mp &< 0. \label{eq: N1 two N2 five 04}
\end{align}
If $z < p \leq 3z$, then it is obvious that the condition of Theorem~1 holds.
If $p > 3z$ or $p < z$, then it follows from \eqref{eq: N1 two N2 five 04} that $\alpha < mp/((p-z)(p-3z))$.
Together with \eqref{eq: N1 two N2 five 02}, one obtains
$mp/z^2 < mp/((p-z)(p-3z))$,
which implies $(2 - \sqrt{2})z < p < (2 + \sqrt{2})z$. Therefore,  the condition of Theorem~1  holds.

If $Z_b(s)$ satisfies Condition~2 of Lemma~\ref{lemma: biquadratic impedances Z2 Five elements Two reactive}, then $\alpha p^2 > \gamma$ and $\alpha^2 p^2 + \beta^2 - 2\alpha \beta p  - \alpha \gamma < 0$. Together with \eqref{eq: N1 two N2 five 01}, one obtains \eqref{eq: N1 two N2 five 03} and
\begin{equation}  \label{eq: N1 two N2 five 05}
(p-z)(p-3z)\alpha^2 + m(3p-4z)\alpha + m^2 < 0.
\end{equation}
It follows from \eqref{eq: N1 two N2 five 05} that $p < 3z$. If $z < p < 3z$,  then it is obvious that the condition of Theorem~1 holds. If $p < z$, then
\eqref{eq: N1 two N2 five 03} implies that
\begin{equation}   \label{eq: N1 two N2 five 06}
\alpha < -\frac{mp}{(p-z)(p+z)}.
\end{equation}
From \eqref{eq: N1 two N2 five 02} and \eqref{eq: N1 two N2 five 06}, it follows that $m/(2z) < -mp/((p-z)(p+z))$,
which implies $p > z/(1+\sqrt{2})$. Therefore, the condition of Theorem~1  holds.

If $Z_b(s)$ satisfies Condition~3 of Lemma~\ref{lemma: biquadratic impedances Z2 Five elements Two reactive}, then $\alpha p^2 < \gamma$ and $3 \alpha p^2 + \gamma - 2 \beta p  < 0$.
Together with \eqref{eq: N1 two N2 five 01}, one obtains
\begin{align}
(p-z)(p+z)\alpha + mp &< 0,  \label{eq: N1 two N2 five 07} \\
(3p-z)(p-z)\alpha + mp &< 0. \label{eq: N1 two N2 five 08}
\end{align}
It follows from \eqref{eq: N1 two N2 five 08} that
$z/3 < p < z$. Therefore, the condition of Theorem~1 holds.

If $Z_b(s)$ satisfies Condition~4 of Lemma~\ref{lemma: biquadratic impedances Z2 Five elements Two reactive}, then $\alpha p^2 < \gamma$ and $\beta^2 p^2 + \gamma^2 - \alpha  \gamma p^2 - 2 \beta \gamma p < 0$. Together with \eqref{eq: N1 two N2 five 01}, one obtains \eqref{eq: N1 two N2 five 07} and
\begin{equation} \label{eq: N1 two N2 five 09}
z^2 (3p - z)(p - z)\alpha^2 +  p^3 m \alpha < 0.
\end{equation}
It follows from \eqref{eq: N1 two N2 five 09} that $z/3 < p < z$. Therefore, the condition of Theorem~1 holds.

This means that $Z(s)$ cannot be realized as in Fig.~\ref{fig: Equivalent structure}(a), where $N_b$ is a two-reactive five-element series-parallel network.
It is clear that any network in Fig.~\ref{fig: Equivalent structure}(b)
can be a frequency inverse network of another one in
Fig.~\ref{fig: Equivalent structure}(a). Therefore, by the principle of frequency inverse (Section~III), $Z(s)$ cannot be realized as in
Fig.~\ref{fig: Equivalent structure}(b), where $N_b$ is a two-reactive five-element series-parallel network.

By Lemma~\ref{lemma: biquadratic impedances Z2 Five elements Two reactive}, $Z(s) \in \mathcal{Z}_{p2,z2}$ cannot be realized as a three-reactive seven-element series-parallel network as shown  in Fig.~2(a), where $N_1$
is a three-element series-parallel network and $N_2$ is a four-element series-parallel network. Since any network in Fig.~2(b), where $N_1$
is a three-element series-parallel network and $N_2$ is a four-element series-parallel network, can be a dual network of the case of Fig.~2(a), by the principle of duality (Section~III), this lemma can be proved.

\section{Supplementary Lemmas of Four-Reactive Seven-Element Series-Parallel Realizations for the Proof of  Lemma~3}

\begin{lemma}
\label{lemma: Possible configurations of N1}
{Consider the four-reactive  seven-element series-parallel network in  Fig.~2,  realizing a biquadratic impedance $Z(s)$ in the form of \eqref{eq: general biquadratic impedances} with $A$, $B$, $C$, $D$, $E$, $F$ $> 0$, where $N_1$ is a two-reactive three-element series-parallel network and $N_2$ is a two-reactive  four-element series-parallel network.
If $Z(s)$ cannot be realized as a series-parallel network containing fewer than seven elements, then $N_1$ must be
one of the configurations in Fig.~\ref{fig: Two-reactive configurations N1}.}
\end{lemma}
\begin{figure}[thpb]
      \centering
      \subfigure[]{
      \includegraphics[scale=1.15]{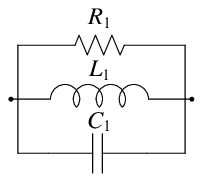}
      \label{fig: Two-reactive configurations N1 a}}
      \subfigure[]{
      \includegraphics[scale=1.15]{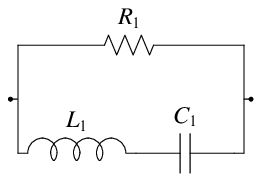}
      \label{fig: Two-reactive configurations N1 b}}
      \subfigure[]{
      \includegraphics[scale=1.15]{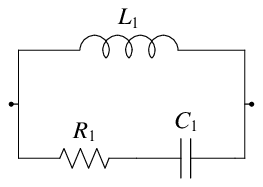}
      \label{fig: Two-reactive configurations N1 c}}
      \subfigure[]{
      \includegraphics[scale=1.15]{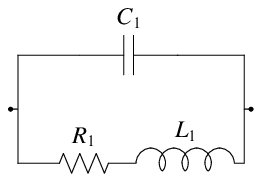}
      \label{fig: Two-reactive configurations N1 d}}
      \caption{Two-reactive three-element configurations for the $N_1$ mentioned in Lemma~\ref{lemma: Possible configurations of N1} (Fig.~8).}
      \label{fig: Two-reactive configurations N1}
\end{figure}
\begin{proof}
Let $\mathcal{C}(a,a')$ denote the \textit{cut-set} \cite[pg.~28]{SR61} separating a one-terminal-pair labeled graph of a network   into two connected subgraphs containing two terminal vertices $a$ and $a'$, respectively.
By \cite[Lemma~1]{WCH14}, for any realization of $Z(s)$ there is no cut-set $\mathcal{C}(a,a')$ corresponding to only one kind of reactive elements, where $a$ and $a'$ denote two terminals.
Since $N_1$ contains three elements, all the possible \textit{network graphs} \cite{CWSL13}
of $N_1$ are listed as in Fig.~\ref{fig: Three-element graphs}. By the method of enumeration, this lemma can be proved.
\end{proof}

\begin{figure}[thpb]
      \centering
      \subfigure[]{
      \includegraphics[scale=1.1]{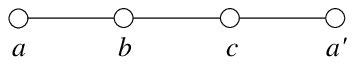}
      \label{fig: Three-element graphs a}}
      \subfigure[]{
      \includegraphics[scale=1.1]{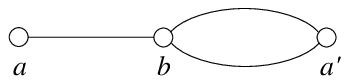}
      \label{fig: Three-element graphs b}} \\
      \subfigure[]{
      \includegraphics[scale=1.1]{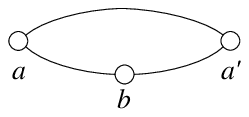}
      \label{fig: Three-element graphs c}}
      \subfigure[]{
      \includegraphics[scale=1.1]{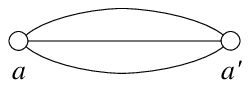}
      \label{fig: Three-element graphs d}}
      \caption{Possible network graphs for three-element networks.}
      \label{fig: Three-element graphs}
\end{figure}

\begin{lemma}  \label{lemma: 01 N1}
{A biquadratic impedance $Z(s) \in \mathcal{Z}_{p2,z2}$ cannot be realized as in Fig.~2(a), where $N_1$ is the configuration in Fig.~\ref{fig: Two-reactive configurations N1}(a)
and $N_2$ is a two-reactive four-element series-parallel network.}
\end{lemma}
\begin{proof}
By calculation,  the impedance of the configuration in Fig.~\ref{fig: Two-reactive configurations N1}(a) is obtained as
\begin{equation*}
Z_1(s) = \frac{R_1L_1 s}{R_1L_1C_1 s^2 + L_1 s + R_1}.
\end{equation*}
Since by assumption the condition of Theorem~1 does not hold, $N_2$ cannot be equivalent to
a series-parallel network containing fewer than four elements. If
$Z(s)$ is realizable as in
Fig.~2(a), where $N_1$ is the configuration in Fig.~\ref{fig: Two-reactive configurations N1}(a) and $N_2$ is a two-reactive four-element series-parallel network, then
the impedance of $N_1$ is in the form of
\begin{equation*}
Z_1(s) = \frac{m s}{(s + p)^2},
\end{equation*}
where $m > 0$, and the impedance of $N_2$ is in the form of
\begin{equation}  \label{eq: 01 N1 00 Z2}
Z_2(s) = \frac{\alpha s^2 + \beta s + \gamma}{(s + p)^2},
\end{equation}
where $\alpha$, $\beta$, $\gamma$ $> 0$, and moreover the condition of Lemma~\ref{lemma: biquadratic impedances Z2 Four elements} holds.
Since
$\alpha s^2 + (\beta + m) s + \gamma = k(s + z)^2$,
it follows that
\begin{align}
\beta + m &= 2 z \alpha, \label{eq: 01 N1 01} \\
\gamma &= z^2 \alpha.  \label{eq: 01 N1 02}
\end{align}

Since $\alpha$, $\gamma$ $> 0$, $Z_2(s)$
satisfies neither Condition~1 nor Condition~2 of Lemma~\ref{lemma: biquadratic impedances Z2 Four elements}.

If $Z_2(s)$ satisfies Condition~3 of Lemma~\ref{lemma: biquadratic impedances Z2 Four elements}, then $\alpha p^2 - \gamma = 0$. Together with \eqref{eq: 01 N1 02}, it follows that $p = z$, which contradicts the assumption.

If $Z_2(s)$ satisfies Condition~4 of Lemma~\ref{lemma: biquadratic impedances Z2 Four elements}, then $\alpha p^2 < \gamma$ and
  either
  $3 \alpha p^2 + \gamma - 2\beta p = 0$ or $\beta^2 p^2 + \gamma^2 - \alpha  \gamma p^2 - 2 \beta \gamma p = 0$ holds. For the case of $3 \alpha p^2 + \gamma - 2\beta p = 0$,
together with \eqref{eq: 01 N1 01} and \eqref{eq: 01 N1 02}, one obtains
$\alpha = - 2mp/((3p-z)(p-z))$, $\beta = -m(3p^2 + z^2)/((3p-z)(p-z))$, and $\gamma = -2mz^2p/((3p-z)(p-z))$,
which implies $z/3 < p < z$ by $\alpha$, $\beta$, $\gamma$ $> 0$.
Thus, the condition of Theorem~1 holds.
For the case of $\beta^2 p^2 + \gamma^2 - \alpha  \gamma p^2 - 2 \beta \gamma p = 0$, together with \eqref{eq: 01 N1 01} and \eqref{eq: 01 N1 02}, one obtains
\begin{equation}   \label{eq: 01 N1 03}
\alpha = \frac{mp}{z(3p-z)}, ~~~ \beta = -\frac{m(p-z)}{3p-z}, ~~~
\gamma = \frac{mzp}{3p - z},
\end{equation}
or
\begin{equation}   \label{eq: 01 N1 04}
\alpha = \frac{mp}{z(p - z)}, ~~~ \beta = \frac{m(p+z)}{p-z}, ~~~
\gamma = \frac{mzp}{p-z}.
\end{equation}
Because
$\alpha$, $\beta$, $\gamma$ $> 0$, it follows from \eqref{eq: 01 N1 03} that $z/3 < p < z$, which satisfies the condition of Theorem~1.
Substituting \eqref{eq: 01 N1 04} into $\alpha p^2 < \gamma$ yields
$(p+z)pm/z < 0$, which is impossible.

If $Z_2(s)$ satisfies Condition~5 of Lemma~\ref{lemma: biquadratic impedances Z2 Four elements}, then $\alpha p^2 > \gamma$ and either $\alpha p^2 + 3 \gamma - 2\beta p = 0$
  or $\alpha^2 p^2 + \beta^2   - 2 \alpha \beta p - \alpha \gamma  = 0$ holds. For the case of $\alpha p^2 + 3 \gamma - 2\beta p = 0$, together with \eqref{eq: 01 N1 01} and \eqref{eq: 01 N1 02}, one
obtains
$\alpha = -2mp/((p-z)(p-3z))$, $\beta = -m(p^2 + 3z^2)/((p-z)(p-3z))$, and $\gamma = -2mz^2p/((p-z)(p-3z))$,
which implies $z < p < 3z$ because $\alpha$, $\beta$, $\gamma$ $> 0$. Thus, the condition of Theorem~1 holds.
For the case of
$\alpha^2 p^2 + \beta^2   - 2 \alpha \beta p - \alpha \gamma  = 0$,
together with \eqref{eq: 01 N1 01} and \eqref{eq: 01 N1 02}, one
obtains
\begin{equation} \label{eq: 01 N1 05}
\alpha = -\frac{m}{p-3z}, ~~~ \beta = -\frac{m(p-z)}{p-3z}, ~~~
\gamma = - \frac{mz^2}{p - 3z},
\end{equation}
or
\begin{equation} \label{eq: 01 N1 06}
\alpha = - \frac{m}{p - z}, ~~~ \beta = - \frac{m(p + z)}{p - z}, ~~~ \gamma = - \frac{mz^2}{p - z}.
\end{equation}
It follows from \eqref{eq: 01 N1 05} that $z < p < 3z$
by
$\alpha$, $\beta$, $\gamma$ $> 0$. Thus, the condition of Theorem~1 holds. Substituting \eqref{eq: 01 N1 06} into $\alpha p^2 > \gamma$ yields $-m(p+z) > 0$, which is impossible.

If $Z_2(s)$ satisfies Condition~6 of Lemma~\ref{lemma: biquadratic impedances Z2 Four elements}, then $\alpha^2 p^4 + 6 \alpha \gamma p^2 + \gamma^2 - 2\alpha \beta p^3 - 2 \beta \gamma p = 0$.
Together with \eqref{eq: 01 N1 01} and \eqref{eq: 01 N1 02}, one
obtains $\alpha = 0$, $\beta = -m$, and $\gamma = 0$ or $\alpha = - 2mp(z^2 + p^2)/(p-z)^4$, $\beta = - m(p^4 + 6z^2p^2 + z^4)/(p-z)^4$, and $\gamma = - 2mz^2p(z^2+p^2)/(p-z)^4$,
which contradicts the assumption that $\alpha$, $\beta$, $\gamma$ $> 0$.

As a conclusion, this lemma is proved.
\end{proof}

\begin{lemma}  \label{lemma: 02 N1}
{If a biquadratic impedance $Z(s) \in \mathcal{Z}_{p2,z2}$ is realizable as in Fig.~2(a), where $N_1$ is one of the configurations in Figs.~\ref{fig: Two-reactive configurations N1}(b), \ref{fig: Two-reactive configurations N1}(c), and \ref{fig: Two-reactive configurations N1}(d),
and $N_2$ is a two-reactive four-element  series-parallel network, then the condition of Lemma~1 holds.}
\end{lemma}
\begin{proof}
First, the case where $N_1$ is a configuration in Fig.~\ref{fig: Two-reactive configurations N1}(b) and $N_2$ is a two-reactive four-element series-parallel network will be discussed.

It is calculated that the impedance of the configuration in Fig.~\ref{fig: Two-reactive configurations N1}(b) is in the form of
\begin{equation} \label{eq: Z1 b configuration}
Z_1(s) = \frac{R_1L_1C_1 s^2 + R_1}{L_1C_1 s^2 + R_1C_1 s + 1}.
\end{equation}
Since it is assumed that the condition of Theorem~1 does not hold, $N_2$ cannot be equivalent to
a series-parallel network containing fewer than four elements. If $Z(s)$ is realizable as in Fig.~2(a), where $N_1$ is the configuration in Fig.~\ref{fig: Two-reactive configurations N1}(b)
and $N_2$ is a two-reactive four-element  series-parallel network,
then the impedance of $N_1$ is in the form of
\begin{equation*}
Z_1(s) = \frac{m (s^2 + p^2)}{(s + p)^2},
\end{equation*}
where $m > 0$, and the impedance of $N_2$ is in the form of \eqref{eq: 01 N1 00 Z2}, where $\beta > 0$, $\alpha$,  $\gamma$ $\geq 0$, and the condition of Lemma~\ref{lemma: biquadratic impedances Z2 Four elements} holds.
Since
$(m + \alpha) s^2 + \beta s + (mp^2 + \gamma) = k (s + z)^2$,
it follows that
\begin{align}
\beta &= 2z(m + \alpha), \label{eq: 02 N1 01} \\
mp^2 + \gamma &= z^2(m + \alpha).  \label{eq: 02 N1 02}
\end{align}

If $Z_2(s)$ satisfies Condition~1 of Lemma~\ref{lemma: biquadratic impedances Z2 Four elements}, then $\beta$, $\gamma$ $> 0$, $\alpha = 0$, and $\gamma < 2 \beta p$. Together with \eqref{eq: 02 N1 01} and \eqref{eq: 02 N1 02}, one obtains
\begin{equation} \label{eq: 02 N1 03}
\beta = 2mz, ~~~ \gamma = -m(p-z)(p+z),
\end{equation}
which implies $p < z$ by $\beta$, $\gamma$ $> 0$.
Substituting \eqref{eq: 02 N1 03} into $\gamma < 2 \beta p$ yields
$m(p^2 + 4zp - z^2) > 0$.
This implies $z/(2 + \sqrt{5}) < p < z$, which satisfies the condition of Lemma~1.

If $Z_2(s)$ satisfies Condition~2 of Lemma~\ref{lemma: biquadratic impedances Z2 Four elements}, then $\alpha$, $\beta$ $> 0$, $\gamma = 0$, and $\alpha p < 2 \beta$. Together with \eqref{eq: 02 N1 01} and \eqref{eq: 02 N1 02}, one obtains
\begin{equation}  \label{eq: 02 N1 04}
\alpha = \frac{m(p-z)(p+z)}{z^2}, ~~~ \beta = \frac{2mp^2}{z},
\end{equation}
which implies $p > z$ by $\alpha$, $\beta$ $> 0$.
Substituting \eqref{eq: 02 N1 04} into $\alpha p < 2 \beta$ yields
$mp(p^2 - 4zp - z^2)/z^2 < 0$.
This further implies $z < p < (2 + \sqrt{5})z$. Therefore, the condition of Lemma~1 holds.

If $Z_2(s)$ satisfies Condition~3 of Lemma~\ref{lemma: biquadratic impedances Z2 Four elements}, then  $\alpha$, $\beta$, $\gamma$ $> 0$  and $\alpha p^2 - \gamma = 0$.  Together with \eqref{eq: 02 N1 01} and \eqref{eq: 02 N1 02}, one obtains $\alpha = -m$, $\beta = 0$, and $\gamma = -mp^2$, which contradicts the assumption.

If $Z_2(s)$ satisfies Condition~4 of Lemma~\ref{lemma: biquadratic impedances Z2 Four elements}, then $\alpha$, $\beta$, $\gamma$ $> 0$, $\alpha p^2 < \gamma$, and
either $3 \alpha p^2 + \gamma - 2\beta p = 0$ or $\beta^2 p^2 + \gamma^2 - \alpha  \gamma p^2 - 2 \beta \gamma p = 0$. For the case of $3 \alpha p^2 + \gamma - 2\beta p = 0$, together with \eqref{eq: 02 N1 01} and \eqref{eq: 02 N1 02}, one obtains
\begin{equation}  \label{eq: 02 N1 05}
\alpha = \frac{m(p^2 + 4zp - z^2)}{(3p-z)(p-z)}, ~~~
\beta = \frac{8mzp^2}{(3p-z)(p-z)}, ~~~
\gamma = - \frac{mp^2(3p^2 - 4zp - 3z^2)}{(3p-z)(p-z)}.
\end{equation}
Substituting \eqref{eq: 02 N1 05} into $\alpha p^2 < \gamma$ yields
$4mp^2(p+z)/(3p-z) < 0$,
which implies $z/(2 + \sqrt{5}) < p < z/3$ together with \eqref{eq: 02 N1 05} by $\alpha$, $\beta$, $\gamma$ $> 0$. Thus, the condition of Lemma~1 holds.
For the case of $\beta^2 p^2 + \gamma^2 - \alpha  \gamma p^2 - 2 \beta \gamma p = 0$, together with \eqref{eq: 02 N1 01} and \eqref{eq: 02 N1 02}, one obtains
\begin{equation}  \label{eq: 02 N1 06}
\alpha = -m, ~~~ \beta = 0, ~~~ \gamma = -mp^2,
\end{equation}
or
\begin{equation}  \label{eq: 02 N1 07}
\alpha = - \frac{m(p^2 + 2zp - z^2)^2}{z^2(3p-z)(p-z)}, ~~~
\beta = - \frac{2mp^2(p^2 + 4zp - z^2)}{z(3p-z)(p-z)}, ~~~
\gamma = - \frac{4mp^4}{(3p-z)(p-z)}.
\end{equation}
By $\alpha$, $\beta$, $\gamma$ $> 0$, \eqref{eq: 02 N1 06} is impossible.  It is implied from \eqref{eq: 02 N1 07} that
$z/3 < p < z$, which satisfies the condition of Lemma~1.

If $Z_2(s)$ satisfies Condition~5 of Lemma~\ref{lemma: biquadratic impedances Z2 Four elements}, then $\alpha$, $\beta$, $\gamma$ $> 0$, $\alpha p^2 > \gamma$, and either $\alpha p^2 + 3 \gamma - 2\beta p = 0$ or
$\alpha^2 p^2 + \beta^2   - 2 \alpha \beta p - \alpha \gamma  = 0$. For the case of  $\alpha p^2 + 3 \gamma - 2\beta p = 0$,
together with \eqref{eq: 02 N1 01} and \eqref{eq: 02 N1 02}, one obtains
\begin{equation} \label{eq: 02 N1 08}
\alpha = \frac{m(3p^2 + 4zp - 3z^2)}{(p-z)(p-3z)}, ~~~
\beta = \frac{8mzp^2}{(p-z)(p-3z)}, ~~~
\gamma = -\frac{mp^2(p^2 - 4zp - z^2)}{(p-z)(p-3z)}.
\end{equation}
Substituting \eqref{eq: 02 N1 08} into $\alpha p^2 > \gamma$ yields
$4mp^2(p+z)/(p-3z) > 0$,
which implies $3z < p < (2 + \sqrt{5})z$ together with \eqref{eq: 02 N1 08} by $\alpha$, $\beta$, $\gamma$ $> 0$.
For the case of $\alpha^2 p^2 + \beta^2   - 2 \alpha \beta p - \alpha \gamma  = 0$, together with \eqref{eq: 02 N1 01} and \eqref{eq: 02 N1 02}, one obtains
\begin{equation}  \label{eq: 02 N1 09}
\alpha = -m, ~~~ \beta = 0, ~~~ \gamma = - m p^2,
\end{equation}
or
\begin{equation}  \label{eq: 02 N1 10}
\alpha = - \frac{4mz^2}{(p-z)(p-3z)}, ~~~
\beta = \frac{2mz(p^2 - 4zp - z^2)}{(p-z)(p-3z)}, ~~~
\gamma = - \frac{m(p^2 - 2zp - z^2)^2}{(p-z)(p-3z)}.
\end{equation}
It is obvious that
\eqref{eq: 02 N1 09} contradicts the assumption that $\alpha$, $\beta$, $\gamma$ $> 0$.
Moreover, it is implied from \eqref{eq: 02 N1 10} that
$z < p < 3z$, which satisfies the condition of Theorem~1.

If $Z_2(s)$ satisfies Condition~6 of Lemma~\ref{lemma: biquadratic impedances Z2 Four elements}, then $\alpha$, $\beta$, $\gamma$ $> 0$ and $\alpha^2 p^4 + 6 \alpha \gamma p^2 + \gamma^2 - 2\alpha \beta p^3 - 2 \beta \gamma p = 0$. Together with \eqref{eq: 02 N1 01} and \eqref{eq: 02 N1 02}, one obtains
\begin{equation} \label{eq: 02 N1 11}
\begin{split}
\alpha &= \frac{m(3p^4 - 2 z^2p^2 + 4z^3p - z^4 + 2p^2\sqrt{2p^4 + 2 z^4})}{(p-z)^4},    \\
\beta &= \frac{4mzp^2(2p^2 - 2zp + 2z^2 +\sqrt{2p^4 + 2z^4})}{(p-z)^4},      \\
\gamma &= \frac{mp^2(-p^4 + 4zp^3 - 2z^2p^2 + 3z^4 + 2z^2\sqrt{2p^4 + 2z^4})}{(p-z)^4},
\end{split}
\end{equation}
or
\begin{equation} \label{eq: 02 N1 12}
\begin{split}
\alpha &= \frac{m(3p^4 - 2 z^2p^2 + 4z^3p - z^4 - 2p^2\sqrt{2p^4 + 2 z^4})}{(p-z)^4},    \\
\beta &= \frac{4mzp^2(2p^2 - 2zp + 2z^2 -\sqrt{2p^4 + 2z^4})}{(p-z)^4},      \\
\gamma &= \frac{mp^2(-p^4 + 4zp^3 - 2z^2p^2 + 3z^4 - 2z^2\sqrt{2p^4 + 2z^4})}{(p-z)^4}.
\end{split}
\end{equation}
Consider the solutions in \eqref{eq: 02 N1 11}.
Assume that $p \geq (2 + \sqrt{5})z$. Then, $\gamma < 0$ since
$-p^4 + 4zp^3 - 2z^2p^2 + 3z^4 < 0$ and
$(2z^2\sqrt{2p^4 + 2z^4})^2 - (-p^4 + 4zp^3 - 2z^2p^2 + 3z^4)^2 = -(p+z)(p^2 - 4zp - z^2)(p-z)^5 \leq 0$. This contradicts the assumption. Assume that $p \leq z/(2 + \sqrt{5})$. Then, $\alpha < 0$ since $3p^4 - 2 z^2p^2 + 4z^3p - z^4 < 0$ and $(2p^2\sqrt{2p^4 + 2z^4})^2 - (3p^4 - 2z^2p^2 + 4z^3p - z^4)^2 = -(p+z)(p^2 + 4zp - z^2)(p-z)^5 \leq 0$. This also contradicts the assumption.
Consider the solutions in \eqref{eq: 02 N1 12}. Assume that $p \geq (2 + \sqrt{5})z$. Then, $\gamma < 0$ because of $-p^4 + 4zp^3 - 2z^2p^2 + 3z^4 < 0$. This contradicts the assumption. Assume that $p \leq z/(2 + \sqrt{5})$. Then, $\alpha < 0$ because of $3p^4 - 2 z^2p^2 + 4z^3p - z^4 < 0$. This also contradicts the assumption.

Therefore, it can be proved that if
a biquadratic impedance $Z(s) \in \mathcal{Z}_{p2,z2}$ is realizable as in Fig.~2(a),
where $N_1$ is one of the configurations in Fig.~\ref{fig: Two-reactive configurations N1}(b)
and $N_2$ is a two-reactive four-element series-parallel network, then the condition of Lemma~1 holds.

Then, it turns to the case where $N_1$ is a configurations in Figs.~\ref{fig: Two-reactive configurations N1}(c) and $N_2$ is a two-reactive four-element series-parallel network.

It is calculated that the impedance of the configuration in Fig.~\ref{fig: Two-reactive configurations N1}(c)  is in the form of
\begin{equation}  \label{eq: Z1 c configuration}
Z_1(s) = \frac{s(R_1L_1C_1 s + L_1)}{L_1C_1 s^2 + R_1C_1 s + 1}.
\end{equation}
Since it is assumed that the condition of Theorem~1 does not hold, $N_2$ cannot be equivalent to
a series-parallel network containing fewer than four elements.
If $Z(s)$ is realizable as in Fig.~2(a),
where $N_1$ is the configuration in Fig.~\ref{fig: Two-reactive configurations N1}(c)
and $N_2$ is a two-reactive four-element series-parallel network, then
the impedance of $N_1$ is in the form of
\begin{equation*}
Z_1(s) = \frac{ms(s + p/2)}{(s + p)^2},
\end{equation*}
where $m > 0$, and the impedance of $N_2$ is in the form of
\eqref{eq: 01 N1 00 Z2},
where  $\beta$, $\gamma$ $> 0$, $\alpha$ $\geq 0$, and  the condition of Lemma~\ref{lemma: biquadratic impedances Z2 Four elements} holds. Since
$(m + \alpha) s^2 + (mp/2 + \beta) s + \gamma
= k(s + z)^2$,
it follows that
\begin{align}
\frac{mp}{2} + \beta &= 2(m+\alpha)z,  \label{eq: 03 N1 01} \\
\gamma &= (m+\alpha)z^2.   \label{eq: 03 N1 02}
\end{align}

If $Z_2(s)$ satisfies Condition~1 of Lemma~\ref{lemma: biquadratic impedances Z2 Four elements}, then
$\beta$, $\gamma$ $> 0$, $\alpha = 0$, and $\gamma < 2 \beta p$. Together with \eqref{eq: 03 N1 01} and \eqref{eq: 03 N1 02}, one obtains
\begin{equation} \label{eq: 03 N1 03}
\beta = - \frac{1}{2}m(p-4z), ~~~ \gamma = mz^2,
\end{equation}
which implies $p < 4z$ by $\beta > 0$.
Substituting \eqref{eq: 03 N1 03} into $\gamma < 2 \beta p$ yields
$m(p^2 - 4zp + z^2) < 0$,
which further implies $z/(2 + \sqrt{3}) < p < (2 + \sqrt{3})z$ together with \eqref{eq: 03 N1 03}. Thus, the condition of Lemma~1 holds.

Since $\gamma > 0$, $Z_2(s)$ cannot satisfy Condition~2 of Lemma~\ref{lemma: biquadratic impedances Z2 Four elements}.

If $Z_2(s)$ satisfies Condition~3 of Lemma~\ref{lemma: biquadratic impedances Z2 Four elements}, then $\alpha$, $\beta$, $\gamma$ $> 0$ and $\alpha p^2 - \gamma = 0$. Together with \eqref{eq: 03 N1 01} and \eqref{eq: 03 N1 02}, one obtains
\begin{equation*}
\alpha = \frac{mz^2}{(p-z)(p+z)}, ~~~
\beta = -\frac{mp(p^2 - 4zp - z^2)}{2(p-z)(p+z)}, ~~~
\gamma = \frac{mz^2p^2}{(p-z)(p+z)},
\end{equation*}
which implies $z < p < (2 + \sqrt{5})z$ by $\alpha$, $\beta$, $\gamma$ $> 0$. Thus, the condition of Lemma~1 holds.

If $Z_2(s)$ satisfies Condition~4 of Lemma~\ref{lemma: biquadratic impedances Z2 Four elements}, then $\alpha$, $\beta$, $\gamma$ $> 0$, $\alpha p^2 < \gamma$, and
  either
  $3 \alpha p^2 + \gamma - 2\beta p = 0$ or $\beta^2 p^2 + \gamma^2 - \alpha  \gamma p^2 - 2 \beta \gamma p = 0$.
For the case of $3 \alpha p^2 + \gamma - 2\beta p = 0$, together with \eqref{eq: 03 N1 01} and \eqref{eq: 03 N1 02}, one obtains
\begin{equation*}
\alpha = - \frac{m(p^2 - 4zp + z^2)}{(3p - z)(p - z)}, ~~~
\beta = - \frac{mp(3p^2 - 12zp + z^2)}{2(3p-z)(p-z)}, ~~~
\gamma = \frac{2mz^2p^2}{(3p-z)(p-z)},
\end{equation*}
which implies $z/(2 + \sqrt{3}) < p < z/3$ or $z < p < (2 + \sqrt{3})z$ by $\alpha$, $\beta$, $\gamma$ $> 0$. Thus, the condition of Lemma~1 holds.
For the case of $\beta^2 p^2 + \gamma^2 - \alpha  \gamma p^2 - 2 \beta \gamma p = 0$, together with \eqref{eq: 03 N1 01} and \eqref{eq: 03 N1 02}, one obtains
\begin{equation} \label{eq: 03 N1 07}
\begin{split}
\alpha &= \frac{m(2p^3 - 8zp^2 + 8z^2p - 2z^3 + p^2\sqrt{(p-z)(p-3z)})}{2z(3p-z)(p-z)},  \\
\beta &= \frac{mp(p^2 - z^2 + 2p\sqrt{(p-z)(p-3z)})}{2(3p-z)(p-z)},  \\
\gamma &= \frac{mzp^2(2p-2z+\sqrt{(p-z)(p-3z)})}{2(3p-z)(p-z)},
\end{split}
\end{equation}
or
\begin{equation} \label{eq: 03 N1 08}
\begin{split}
\alpha &= \frac{m(2p^3 - 8zp^2 + 8z^2p - 2z^3 - p^2\sqrt{(p-z)(p-3z)} )}{2z(3p-z)(p-z)},  \\
\beta &=  \frac{mp(p^2 - z^2 - 2p\sqrt{(p-z)(p-3z)})}{2(3p-z)(p-z)},   \\
\gamma &= \frac{2p - 2z - \sqrt{(p-z)(p-3z)}}{2(3p-z)(p-z)},
\end{split}
\end{equation}
where $p < z$ or $p > 3z$ must hold to guarantee the existence of the solutions.
Consider the solutions in \eqref{eq: 03 N1 07}. Substituting \eqref{eq: 03 N1 07} into $\alpha p^2 < \gamma$ yields
\begin{equation}  \label{eq: 03 N1 09}
\frac{mp^2(2p(p-3z) + (p+z)\sqrt{(p-z)(p-3z)})}{2(3p-z)z} < 0.
\end{equation}
It is further implied from \eqref{eq: 03 N1 09} that $p < z$. Moreover, since $((p+z)\sqrt{(p-z)(p-3z)})^2 - (2p^2 - 6zp)^2 = -(3p-z)(p-3z)(p^2 - 4zp - z^2)$, it is implied that $z/3 < p < z$, which satisfies  the condition of Theorem~1.
Consider the solutions in \eqref{eq: 03 N1 08}. The assumption that $\beta$, $\gamma$ $> 0$ implies $z/3 < p < z$ or $p > 3z$. Assume that $p \geq (2 + \sqrt{5})z$. Then, $\beta \leq 0$ since $p^2 - z^2 > 0$ and $(p^2 - z^2)^2 - (2p\sqrt{(p-z)(p-3z)})^2 =
-(p-z)(3p-z)(p^2 - 4zp - z^2) \leq 0$.
This contradicts the assumption.

If $Z_2(s)$ satisfies Condition~5 of Lemma~\ref{lemma: biquadratic impedances Z2 Four elements}, then
$\alpha$, $\beta$, $\gamma$ $> 0$, $\alpha p^2 > \gamma$, and either $\alpha p^2 + 3 \gamma - 2\beta p = 0$
  or $\alpha^2 p^2 + \beta^2   - 2 \alpha \beta p - \alpha \gamma  = 0$.
For the case of $\alpha p^2 + 3 \gamma - 2\beta p = 0$, together with \eqref{eq: 03 N1 01} and \eqref{eq: 03 N1 02}, one obtains
\begin{equation}  \label{eq: 03 N1 12}
\begin{split}
\alpha &= \frac{m(-p^2 + 6zp - 7z^2 + \sqrt{-z^2(p^2 - 4zp - z^2)})}{2(p-z)(p-3z)},   \\
\beta &= \frac{m(-p^3 + 6zp^2 - 7z^2p - 2z^3 + 2z\sqrt{-z^2(p^2 - 4zp - z^2)})}{2(p-z)(p-3z)},   \\
\gamma &= \frac{mz^2(p^2 - 2zp - z^2 + \sqrt{-z^2(p^2 - 4zp - z^2)})}{2(p-z)(p-3z)},
\end{split}
\end{equation}
or
\begin{equation}  \label{eq: 03 N1 13}
\begin{split}
\alpha &= \frac{m(-p^2 + 6zp - 7z^2 - \sqrt{-z^2(p^2 - 4zp - z^2)})}{2(p-z)(p-3z)},   \\
\beta &= \frac{m(-p^3 + 6zp^2 - 7z^2p - 2z^3 - 2z\sqrt{-z^2(p^2 - 4zp - z^2)})}{2(p-z)(p-3z)},   \\
\gamma &= \frac{mz^2(p^2 - 2zp - z^2 - \sqrt{-z^2(p^2 - 4zp - z^2)})}{2(p-z)(p-3z)},
\end{split}
\end{equation}
where $p \leq (2 + \sqrt{5})z$ must hold to guarantee the existence of the solutions. Assume that $p = (2 + \sqrt{5})z$. It is implied from \eqref{eq: 03 N1 12} and \eqref{eq: 03 N1 13} that $\beta = 0$, which contradicts the assumption.
Assume that $p \leq z/(2 + \sqrt{5})$. For \eqref{eq: 03 N1 12},  it is derived that $\alpha < 0$ since $(-z^2(p^2 - 4zp - z^2))^2 - (-p^2 + 6zp - 7z^2)^2 = -(p-z)(p-3z)(p-4z)^2 < 0$. This contradicts the assumption. For \eqref{eq: 03 N1 13}, it is derived that $\alpha$, $\beta$, $\gamma$ $< 0$ since $-p^2 + 6zp - 7z^2 < 0$, $-p^3 + 6zp^2 - 7z^2p - 2z^3 < 0$, and $p^2 - 2zp - z^2 < 0$.
This also contradicts the assumption.

If $Z_2(s)$ satisfies Condition~6 of Lemma~\ref{lemma: biquadratic impedances Z2 Four elements}, then $\alpha$, $\beta$, $\gamma$ $> 0$ and $\alpha^2 p^4 + 6 \alpha \gamma p^2 + \gamma^2 - 2\alpha \beta p^3 - 2 \beta \gamma p = 0$.
Together with \eqref{eq: 03 N1 01} and \eqref{eq: 03 N1 02}, one obtains
\begin{equation}  \label{eq: 03 N1 14}
\alpha = -m, ~~~ \beta = -\frac{mp}{2}, ~~~ \gamma = 0,
\end{equation}
or
\begin{equation}  \label{eq: 03 N1 15}
\begin{split}
\alpha  &= - \frac{mz^2(p^2 - 4zp + z^2)}{(p-z)^4}, \\
\beta  &= - \frac{mp(p^4 - 8zp^3 + 22z^2p^2 - 24z^3p + z^4)}{2(p-z)^4},  \\
\gamma  &= \frac{mz^2p^2(p^2 - 4zp + 5z^2)}{(p-z)^4}.
\end{split}
\end{equation}
It is obvious that \eqref{eq: 03 N1 14} is impossible. For \eqref{eq: 03 N1 15}, one implies $z/(2 + \sqrt{3}) < p < (2 + \sqrt{3})z$ by $\alpha$, $\beta$, $\gamma$ $> 0$. Thus, the condition of Lemma~1 holds.

Therefore, it can be proved that if
a biquadratic impedance $Z(s) \in \mathcal{Z}_{p2,z2}$ is realizable as in Fig.~2(a),
where $N_1$ is one of the configurations in Fig.~\ref{fig: Two-reactive configurations N1}(c)
and $N_2$ is a two-reactive four-element series-parallel network, then the condition of Lemma~1 holds.

It is clear that  any network in Fig.~2(a), where $N_1$ is a configuration in Fig.~\ref{fig: Two-reactive configurations N1}(d) and $N_2$ is a two-reactive four-element series-parallel network can be a frequency inverse network of the case where $N_1$ is a configuration in Fig.~\ref{fig: Two-reactive configurations N1}(c). By the principle of frequency inverse,
if
a biquadratic impedance $Z(s) \in \mathcal{Z}_{p2,z2}$ is realizable as in Fig.~2(a),
where $N_1$ is one of the configurations in Fig.~\ref{fig: Two-reactive configurations N1}(d)
and $N_2$ is a two-reactive four-element series-parallel network, then the condition of Lemma~1 holds.
\end{proof}

\begin{lemma}
\label{lemma: Possible configurations of N1 one-reactive three-element and N2 three-reactive four-element}
{Consider the four-reactive  seven-element series-parallel network in  Fig.~2(a),  realizing a biquadratic impedance $Z(s)$ in the form of \eqref{eq: general biquadratic impedances} with $A$, $B$, $C$, $D$, $E$, $F$ $> 0$, where $N_1$ is a one-reactive three-element series-parallel network and $N_2$ is a three-reactive  four-element series-parallel network. If $Z(s)$ cannot be realized as a series-parallel network containing fewer than seven elements, then $N_1$ will be equivalent to one of the  configurations in Fig.~\ref{fig: One-reactive Three-element N1}, and $N_2$ will be equivalent to one of the configurations in  Fig.~\ref{fig: Three-reactive Four-element N2}.}
\end{lemma}
\begin{proof}
For any realization of $Z(s)$,
there is no cut-set $\mathcal{C}(a,a')$ corresponding to one kind of reactive elements, where $a$ and $a'$ denote two terminal vertices,
by \cite[Lemma~1]{WCH14}.
The possible network graphs for subnetworks $N_1$ and $N_2$ are listed as in Figs.~\ref{fig: Three-element graphs} and \ref{fig: Four-element graphs}, respectively. Based on the method of enumeration and the equivalence in \cite[Lemma~11]{JS11}, $N_1$ can be equivalent to one of configurations in Fig.~\ref{fig: One-reactive Three-element N1}, and $N_2$ can be equivalent to one of configurations in Fig.~\ref{fig: Three-reactive Four-element N2}.
\end{proof}

\begin{figure}[thpb]
      \centering
      \subfigure[]{
      \includegraphics[scale=1.15]{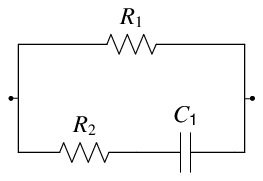}
      \label{fig: One-reactive Three-element N1 a}}
      \subfigure[]{
      \includegraphics[scale=1.15]{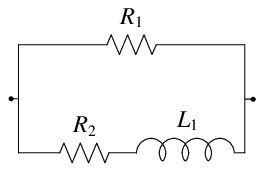}
      \label{fig: One-reactive Three-element N1 b}}
      \caption{One-reactive three-element series-parallel configurations for the $N_1$ mentioned in Lemma~\ref{lemma: Possible configurations of N1 one-reactive three-element and N2 three-reactive four-element}.}
      \label{fig: One-reactive Three-element N1}
\end{figure}

\begin{figure}[thpb]
      \centering
      \subfigure[]{
      \includegraphics[scale=1.15]{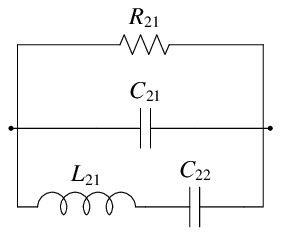}
      \label{fig: Three-reactive Four-element N2 a}}
      \subfigure[]{
      \includegraphics[scale=1.15]{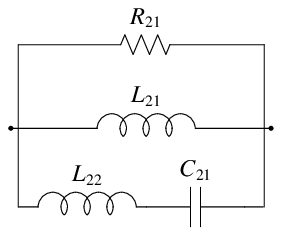}
      \label{fig: Three-reactive Four-element N2 b}}
      \subfigure[]{
      \includegraphics[scale=1.15]{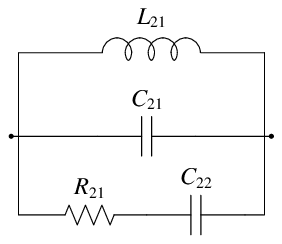}
      \label{fig: Three-reactive Four-element N2 c}}
      \subfigure[]{
      \includegraphics[scale=1.15]{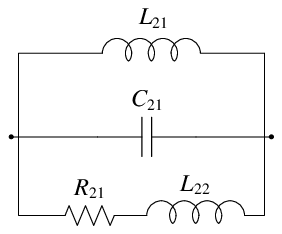}
      \label{fig: Three-reactive Four-element N2 d}}
      \subfigure[]{
      \includegraphics[scale=1.15]{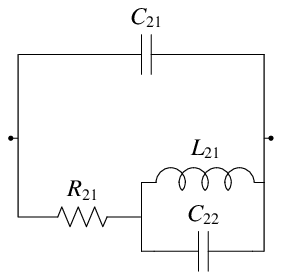}
      \label{fig: Three-reactive Four-element N2 e}}
      \subfigure[]{
      \includegraphics[scale=1.15]{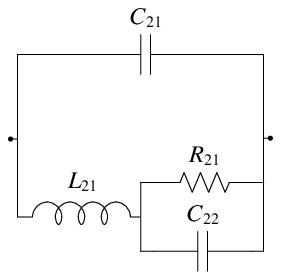}
      \label{fig: Three-reactive Four-element N2 f}}
      \subfigure[]{
      \includegraphics[scale=1.15]{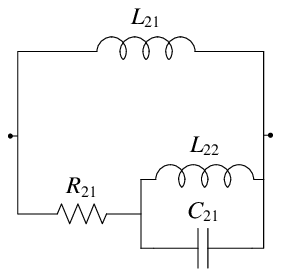}
      \label{fig: Three-reactive Four-element N2 g}}
      \subfigure[]{
      \includegraphics[scale=1.15]{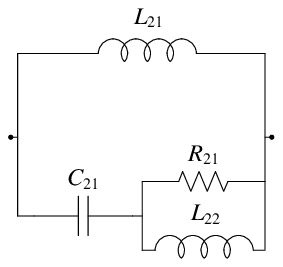}
      \label{fig: Three-reactive Four-element N2 h}}
      \caption{Three-reactive four-element series-parallel configurations for the $N_2$ mentioned in Lemma~\ref{lemma: Possible configurations of N1 one-reactive three-element and N2 three-reactive four-element}.}
      \label{fig: Three-reactive Four-element N2}
\end{figure}

\begin{figure}[thpb]
      \centering
      \subfigure[]{
      \includegraphics[scale=1.15]{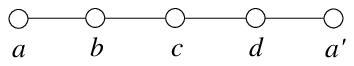}
      \label{fig: Four-element graphs a}}
      \subfigure[]{
      \includegraphics[scale=1.15]{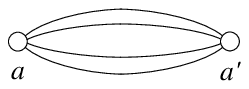}
      \label{fig: Four-element graphs b}} \\
      \subfigure[]{
      \includegraphics[scale=1.15]{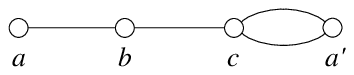}
      \label{fig: Four-element graphs c}}
      \subfigure[]{
      \includegraphics[scale=1.15]{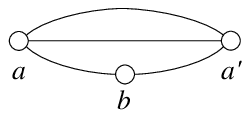}
      \label{fig: Four-element graphs d}}  \\
      \subfigure[]{
      \includegraphics[scale=1.15]{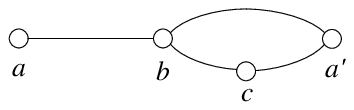}
      \label{fig: Four-element graphs e}}
      \subfigure[]{
      \includegraphics[scale=1.15]{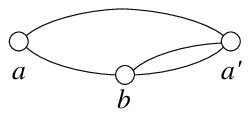}
      \label{fig: Four-element graphs f}}  \\
      \subfigure[]{
      \includegraphics[scale=1.15]{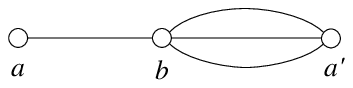}
      \label{fig: Four-element graphs g}}
      \subfigure[]{
      \includegraphics[scale=1.15]{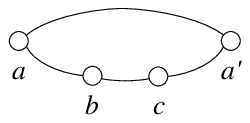}
      \label{fig: Four-element graphs h}}  \\
      \subfigure[]{
      \includegraphics[scale=1.15]{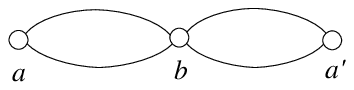}
      \label{fig: Four-element graphs i}}
      \subfigure[]{
      \includegraphics[scale=1.15]{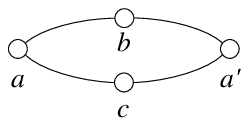}
      \label{fig: Four-element graphs j}}
      \caption{Possible network graphs for four-element networks.}
      \label{fig: Four-element graphs}
\end{figure}

\begin{lemma} \label{lemma: four-reactive seven-element configurations not realizable}
A biquadratic impedance $Z(s) \in \mathcal{Z}_{p2,z2}$  cannot be realized as  in Fig.~2(a), where $N_1$ is the configuration in Fig.~\ref{fig: One-reactive Three-element N1}(a),
and $N_2$ is one of the configurations in Figs.~\ref{fig: Three-reactive Four-element N2}(a) and \ref{fig: Three-reactive Four-element N2}(c)--\ref{fig: Three-reactive Four-element N2}(e).
\end{lemma}
\begin{proof}
By calculation, the impedance of the configuration in Fig.~\ref{fig: Three-reactive Four-element N2}(a) is obtained as
\begin{equation} \label{eq: 01 Noa configurations not realizable}
Z_2(s) = \frac{R_{21}L_{21}C_{22} s^2 + R_{21}}{R_{21}L_{21}C_{21}C_{22} s^3 + L_{21}C_{22} s^2 + R_{21}(C_{21} + C_{22}) s + 1}.
\end{equation}
If $Z(s)$ is realizable as  in Fig.~2(a), where $N_1$ is the configuration in Fig.~\ref{fig: One-reactive Three-element N1}(a) and $N_2$ is the configuration in Fig.~\ref{fig: Three-reactive Four-element N2}(a), then the impedance of $N_1$ is in the form of
\begin{equation} \label{eq: Z1 one-reactive three-element N1}
Z_1(s) = \frac{ms + q}{s + p_1},
\end{equation}
where $m$, $q$, $p_1$ $> 0$  and
\begin{equation} \label{eq: Realizability of N1 one-reactive three-element}
q - mp_1 > 0
\end{equation}
holds, and moreover the impedance of $N_2$ is in the form of
\begin{equation}  \label{eq: 02 Noa configurations not realizable}
Z_2(s) = \frac{\alpha s^2 + \gamma}{(s+p_1)(s+p)^2},
\end{equation}
where $\alpha$, $\gamma$ $> 0$. Then, it follows from \eqref{eq: 01 Noa configurations not realizable} and \eqref{eq: 02 Noa configurations not realizable} that
\begin{align}
\frac{1}{C_{21}} &= \alpha,
\label{eq: 03 Noa configurations not realizable} \\
\frac{1}{L_{21}C_{21}C_{22}} &= \gamma,
\label{eq: 04 Noa configurations not realizable}  \\
\frac{1}{R_{21} C_{21}} &=  2p + p_1,
\label{eq: 05 Noa configurations not realizable}  \\
\frac{C_{21} + C_{22}}{L_{21} C_{21} C_{22}} &= p(p+2p_1),
\label{eq: 06 Noa configurations not realizable}  \\
\frac{1}{R_{21} L_{21} C_{21} C_{22}} &= p_1p^2.
\label{eq: 07 Noa configurations not realizable}
\end{align}
From \eqref{eq: 03 Noa configurations not realizable}, one obtains
\begin{equation} \label{eq: C21 Noa configurations not realizable}
C_{21} = \frac{1}{\alpha}.
\end{equation}
It follows from \eqref{eq: 05 Noa configurations not realizable} and \eqref{eq: C21 Noa configurations not realizable} that
\begin{equation}  \label{eq: R21 Noa configurations not realizable}
R_{21} = \frac{\alpha}{2p+p_1}.
\end{equation}
By \eqref{eq: 04 Noa configurations not realizable}, \eqref{eq: 06 Noa configurations not realizable}, and \eqref{eq: C21 Noa configurations not realizable}, one obtains
\begin{equation}  \label{eq: C22 Noa configurations not realizable}
C_{22} = \frac{\alpha p^2 + 2\alpha p_1 p - \gamma}{\alpha \gamma}.
\end{equation}
By \eqref{eq: 04 Noa configurations not realizable}, \eqref{eq: C21 Noa configurations not realizable}, and \eqref{eq: C22 Noa configurations not realizable}, one obtains
\begin{equation}  \label{eq: L21 Noa configurations not realizable}
L_{21} = \frac{\alpha^2}{\alpha p^2 + 2 \alpha p_1 p - \gamma}.
\end{equation}
Substituting \eqref{eq: C21 Noa configurations not realizable}--\eqref{eq: L21 Noa configurations not realizable} into \eqref{eq: 07 Noa configurations not realizable} gives
\begin{equation}  \label{eq: Condition01 Noa configurations not realizable}
\alpha p_1 p^2 - 2\gamma p - \gamma p_1 = 0.
\end{equation}
The assumption that $C_{22} > 0$ and $L_{21} > 0$ gives
\begin{equation}  \label{eq: Condition02 Noa configurations not realizable}
\alpha p^2 + 2\alpha p_1 p - \gamma > 0.
\end{equation}
Based on \eqref{eq: Z1 one-reactive three-element N1} and \eqref{eq: 02 Noa configurations not realizable}, calculation yields
\begin{equation} \label{eq: Z1 + Z2 Noa configurations not realizable}
\begin{split}
&Z(s) = Z_1(s) + Z_2(s) = \\
&\frac{ms^3 + (2mp + \alpha + q)s^2 + p(mp + 2q)s + (qp^2 + \gamma)}{(s+p_1)(s+p)^2}.
\end{split}
\end{equation}
Comparing (1) with \eqref{eq: Z1 + Z2 Noa configurations not realizable}, one obtains
\begin{align}
m &= k, \label{eq: 101 Noa configurations not realizable}  \\
2mp + \alpha + q &= k(p_1 + 2z),   \label{eq: 102 Noa configurations not realizable}  \\
p(mp + 2q) &= kz(z+2p_1),  \label{eq: 103 Noa configurations not realizable}  \\
qp^2 + \gamma &= kp_1z^2.  \label{eq: 104 Noa configurations not realizable}
\end{align}
Then, \eqref{eq: 101 Noa configurations not realizable} and \eqref{eq: 103 Noa configurations not realizable} together  yield
\begin{equation} \label{eq: q Noa configurations not realizable}
q = \frac{-k(p^2 - 2p_1z - z^2)}{2p}.
\end{equation}
By \eqref{eq: 101 Noa configurations not realizable}, \eqref{eq: 102 Noa configurations not realizable}, and \eqref{eq: q Noa configurations not realizable}, one obtains
\begin{equation} \label{eq: alpha Noa configurations not realizable}
\alpha = \frac{-k(p-z)(3p - z - 2p_1)}{2p}.
\end{equation}
By \eqref{eq: 101 Noa configurations not realizable} and \eqref{eq: 104 Noa configurations not realizable}, one obtains
\begin{equation}  \label{eq: gamma Noa configurations not realizable}
\gamma = \frac{1}{2}k(p-z)(p^2 + zp - 2zp_1).
\end{equation}
Together with \eqref{eq: 101 Noa configurations not realizable} and \eqref{eq: q Noa configurations not realizable}--\eqref{eq: gamma Noa configurations not realizable}, condition~\eqref{eq: Realizability of N1 one-reactive three-element} is equivalent to
$p < z$, and \eqref{eq: Condition01 Noa configurations not realizable} and \eqref{eq: Condition02 Noa configurations not realizable} are equivalent to
\begin{align}
(p+z)p_1^2 - 2p(p-z)p_1 - p^2(p+z) &= 0,
\label{eq: Equivalent Condition01 Noa configurations not realizable}    \\
p^2 + p_1p - (p_1^2 + zp_1) &> 0,
\label{eq: Equivalent Condition02 Noa configurations not realizable}
\end{align}
respectively.
Then, it follows from \eqref{eq: Equivalent Condition01 Noa configurations not realizable} and \eqref{eq: Equivalent Condition02 Noa configurations not realizable} that $-p_1(p-z)^2/(p+z) > 0$, which is impossible.

Therefore, $Z(s) \in \mathcal{Z}_{p2,z2}$  cannot be realized as  in Fig.~2(a), where $N_1$ is the configuration in Fig.~\ref{fig: One-reactive Three-element N1}(a),
and $N_2$ is  the configurations in Fig.~\ref{fig: Three-reactive Four-element N2}(a).

It is calculated that the impedance of the configuration in Fig.~\ref{fig: Three-reactive Four-element N2}(c) is
\begin{equation} \label{eq: 01 Noc configurations not realizable}
Z_2(s) = \frac{R_{21} L_{21} C_{22} s^2 + L_{21} s}{R_{21} L_{21} C_{21} C_{22} s^3 + L_{21} (C_{21} + C_{22}) s^2 + R_{21} C_{22} s + 1}.
\end{equation}
If $Z(s)$ is realizable as  in Fig.~2(a), where $N_1$ is the configuration in Fig.~\ref{fig: One-reactive Three-element N1}(a)
and $N_2$ is the configuration in
Fig.~\ref{fig: Three-reactive Four-element N2}(c), then the impedance of $N_1$ is in the form of \eqref{eq: Z1 one-reactive three-element N1}, where $m$, $q$, $p_1$ $> 0$  and \eqref{eq: Realizability of N1 one-reactive three-element} holds, and the impedance of $N_2$ is in the form of
\begin{equation}  \label{eq: 02 Noc configurations not realizable}
Z_2(s) = \frac{\alpha s^2 + \beta s}{(s+p_1)(s+p)^2},
\end{equation}
where $\alpha$, $\beta$ $> 0$.  Then, it follows from \eqref{eq: 01 Noc configurations not realizable} and \eqref{eq: 02 Noc configurations not realizable} that
\begin{align}
\frac{1}{C_{21}} &= \alpha,
\label{eq: 03 Noc configurations not realizable} \\
\frac{1}{R_{21} C_{21} C_{22}} &= \beta,
\label{eq: 04 Noc configurations not realizable} \\
\frac{C_{21} + C_{22}}{R_{21} C_{21} C_{22}} &= 2p + p_1,
\label{eq: 05 Noc configurations not realizable} \\
\frac{1}{L_{21} C_{21}} &= p(p+2p_1),
\label{eq: 06 Noc configurations not realizable} \\
\frac{1}{R_{21} L_{21} C_{21} C_{22}} &= p_1p^2.
\label{eq: 07 Noc configurations not realizable}
\end{align}
It follows from \eqref{eq: 03 Noc configurations not realizable} that
\begin{equation} \label{eq: C21 Noc configurations not realizable}
C_{21} = \frac{1}{\alpha}.
\end{equation}
By \eqref{eq: 06 Noc configurations not realizable} and \eqref{eq: C21 Noc configurations not realizable}, one obtains
\begin{equation} \label{eq: L21 Noc configurations not realizable}
L_{21} = \frac{\alpha}{p(p+2p_1)}.
\end{equation}
By \eqref{eq: 04 Noc configurations not realizable}, \eqref{eq: 05 Noc configurations not realizable}, and \eqref{eq: C21 Noc configurations not realizable}, one obtains
\begin{equation}  \label{eq: C22 Noc configurations not realizable}
C_{22} = \frac{2\alpha p + \alpha p_1 - \beta}{\alpha \beta}.
\end{equation}
Then, \eqref{eq: 04 Noc configurations not realizable}, \eqref{eq: C21 Noc configurations not realizable}, and \eqref{eq: C22 Noc configurations not realizable} yield
\begin{equation}  \label{eq: R21 Noc configurations not realizable}
R_{21} = \frac{\alpha^2}{2\alpha p + \alpha p_1 - \beta}.
\end{equation}
Then, it follows from \eqref{eq: 07 Noc configurations not realizable}--\eqref{eq: R21 Noc configurations not realizable} that
\begin{equation} \label{eq: Condition01 Noc configurations not realizable}
(\beta - \alpha p_1)p + 2\beta p_1 = 0.
\end{equation}
The assumption that
$R_{21} > 0$ and $C_{22} > 0$ gives
\begin{equation} \label{eq: Condition02 Noc configurations not realizable}
2\alpha p + \alpha p_1 - \beta > 0.
\end{equation}
By \eqref{eq: Z1 one-reactive three-element N1} and \eqref{eq: 02 Noc configurations not realizable}, $Z(s)$ is calculated as
\begin{equation} \label{eq: Z1 + Z2 Noc configurations not realizable}
Z(s) = Z_1(s) + Z_2(s) =
\frac{ms^3 + (2mp + \alpha + q)s^2 + (mp^2 + 2qp + \beta)s + qp^2}{(s+p_1)(s+p)^2}.
\end{equation}
Comparing (1) with \eqref{eq: Z1 + Z2 Noc configurations not realizable}, one obtains
\begin{align}
m  &= k, \label{eq: 101 Noc configurations not realizable}  \\
2mp + \alpha + q &= k(p_1 + 2z),  \label{eq: 102 Noc configurations not realizable}   \\
mp^2 + 2qp + \beta &= kz(z+2p_1),  \label{eq: 103 Noc configurations not realizable}   \\
qp^2 &= kp_1z^2. \label{eq: 104 Noc configurations not realizable}
\end{align}
Then, \eqref{eq: 104 Noc configurations not realizable} yields
\begin{equation} \label{eq: q Noc configurations not realizable}
q = \frac{kp_1z^2}{p^2}.
\end{equation}
By \eqref{eq: 101 Noc configurations not realizable}, \eqref{eq: 102 Noc configurations not realizable}, and \eqref{eq: q Noc configurations not realizable}, one obtains
\begin{equation} \label{eq: alpha Noc configurations not realizable}
\alpha = -\frac{k(p-z)(2p^2-p_1(p+z))}{p^2}.
\end{equation}
It follows from \eqref{eq: 101 Noc configurations not realizable}, \eqref{eq: 103 Noc configurations not realizable}, and \eqref{eq: q Noc configurations not realizable} that
\begin{equation}  \label{eq: beta Noc configurations not realizable}
\beta = - \frac{k(p-z)(p^2 + z(p-2p_1))}{p}.
\end{equation}
Together with \eqref{eq: 101 Noc configurations not realizable} and \eqref{eq: q Noc configurations not realizable}--\eqref{eq: beta Noc configurations not realizable}, condition~\eqref{eq: Realizability of N1 one-reactive three-element} is equivalent to $p < z$, and  \eqref{eq: Condition01 Noc configurations not realizable} and $\alpha > 0$ are equivalent to
\begin{equation}  \label{eq: Equivalent Condition01 Noc configurations not realizable}
(p - 3z)p_1^2 + p^2(p+z) = 0
\end{equation}
and
\begin{equation}  \label{eq: alpha positive Noc configurations not realizable}
p_1 < \frac{2p^2}{p+z},
\end{equation}
respectively. Combining \eqref{eq: Equivalent Condition01 Noc configurations not realizable} and \eqref{eq: alpha positive Noc configurations not realizable}, one obtains $p^2(5p+z)(p-z)^2/(p+z)^2 < 0$, which is impossible.

Therefore, $Z(s) \in \mathcal{Z}_{p2,z2}$  cannot be realized as  in Fig.~2(a), where $N_1$ is the configuration in Fig.~\ref{fig: One-reactive Three-element N1}(a),
and $N_2$ is  the configurations in Fig.~\ref{fig: Three-reactive Four-element N2}(c).

It is calculated that the impedance of the configuration in Fig.~\ref{fig: Three-reactive Four-element N2}(d) is
\begin{equation} \label{eq: 01 Nod configurations not realizable}
Z_2(s) = \frac{L_{21} L_{22} s^2 + R_{21} L_{21} s}{L_{21} L_{22} C_{21} s^3 + R_{21} L_{21} C_{21} s^2 + (L_{21} + L_{22}) s + R_{21}}.
\end{equation}
If $Z(s)$ is realizable as in Fig.~2(a), where $N_1$ is the configuration in Fig.~\ref{fig: One-reactive Three-element N1}(a) and $N_2$ is the configuration in Fig.~\ref{fig: Three-reactive Four-element N2}(d), then the impedance of $N_1$ is in the form of \eqref{eq: Z1 one-reactive three-element N1}, where $m$, $q$, $p_1$ $> 0$  and \eqref{eq: Realizability of N1 one-reactive three-element} holds, and the impedance of $N_2$ is in the form of
\eqref{eq: 02 Noc configurations not realizable} where $\alpha$, $\beta$ $> 0$. Then, it follows from \eqref{eq: 02 Noc configurations not realizable} and \eqref{eq: 01 Nod configurations not realizable} that
\begin{align}
\frac{1}{C_{21}} &= \alpha, \label{eq: 03 Nod configurations not realizable} \\
\frac{R_{21}}{L_{22} C_{21}} &= \beta, \label{eq: 04 Nod configurations not realizable} \\
\frac{R_{21}}{L_{22}} &= 2p + p_1, \label{eq: 05 Nod configurations not realizable} \\
\frac{L_{21}+L_{22}}{L_{21} L_{22} C_{21}} &= p(p+2p_1), \label{eq: 06 Nod configurations not realizable} \\
\frac{R_{21}}{L_{21} L_{22} C_{21}} &= p_1p^2. \label{eq: 07 Nod configurations not realizable}
\end{align}
It follows from \eqref{eq: 03 Nod configurations not realizable} that
\begin{equation} \label{eq: C21 Nod configurations not realizable}
C_{21} = \frac{1}{\alpha}.
\end{equation}
By \eqref{eq: 04 Nod configurations not realizable}, \eqref{eq: 05 Nod configurations not realizable}, and \eqref{eq: C21 Nod configurations not realizable}, one obtains
\begin{equation} \label{eq: Condition01 Nod configurations not realizable}
2 \alpha p + \alpha p_1 - \beta = 0.
\end{equation}
Then, \eqref{eq: 05 Nod configurations not realizable}, \eqref{eq: 07 Nod configurations not realizable}, and \eqref{eq: C21 Nod configurations not realizable} yield
\begin{equation}  \label{eq: L21 Nod configurations not realizable}
L_{21} = \frac{\alpha (2p + p_1)}{p_1 p^2}.
\end{equation}
It follows from \eqref{eq: 06 Nod configurations not realizable}, \eqref{eq: C21 Nod configurations not realizable}, and \eqref{eq: L21 Nod configurations not realizable} that
\begin{equation}  \label{eq: L22 Nod configurations not realizable}
L_{22} = \frac{\alpha (2p + p_1)}{2p(p+p_1)^2}.
\end{equation}
By \eqref{eq: 05 Nod configurations not realizable} and \eqref{eq: L22 Nod configurations not realizable}, one obtains
\begin{equation}
R_{21} = \frac{\alpha (2p + p_1)^2}{2p(p+p_1)^2}.
\end{equation}
By \eqref{eq: Z1 one-reactive three-element N1} and \eqref{eq: 02 Noc configurations not realizable}, $Z(s)$ is calculated as
\eqref{eq: Z1 + Z2 Noc configurations not realizable}. Then, one obtains  \eqref{eq: 101 Noc configurations not realizable}--\eqref{eq: beta Noc configurations not realizable}. Furthermore, it follows from conditions~\eqref{eq: Realizability of N1 one-reactive three-element} and \eqref{eq: Condition01 Nod configurations not realizable} that $z/3 < p < z$, which satisfies the condition of Theorem~1.

Therefore, $Z(s) \in \mathcal{Z}_{p2,z2}$  cannot be realized as  in Fig.~2(a), where $N_1$ is the configuration in Fig.~\ref{fig: One-reactive Three-element N1}(a),
and $N_2$ is  the configurations in Fig.~\ref{fig: Three-reactive Four-element N2}(d).

It is calculated that the impedance of the configuration in Fig.~\ref{fig: Three-reactive Four-element N2}(e) is
\begin{equation} \label{eq: 01 Noe configurations not realizable}
Z_2(s) = \frac{R_{21} L_{21} C_{22} s^2 + L_{21} s + R_{21}}{R_{21} L_{21} C_{21} C_{22} s^3 + L_{21} (C_{21} + C_{22}) s^2 + R_{21} C_{21} s + 1}.
\end{equation}
If $Z(s)$ is realizable as  in Fig.~2(a), where $N_1$ is the configuration in Fig.~\ref{fig: One-reactive Three-element N1}(a) and $N_2$ is the configuration in Fig.~\ref{fig: Three-reactive Four-element N2}(e), then the impedance of $N_1$ is in the form of \eqref{eq: Z1 one-reactive three-element N1}, where $m$, $q$, $p_1$ $> 0$  and \eqref{eq: Realizability of N1 one-reactive three-element} holds, and the impedance of $N_2$ is in the form of
\begin{equation}  \label{eq: 02 Noe configurations not realizable}
Z_2(s) = \frac{\alpha s^2 + \beta s + \gamma}{(s+p_1)(s+p)^2},
\end{equation}
where $\alpha$, $\beta$, $\gamma$ $> 0$. Then, it follows from \eqref{eq: 01 Noe configurations not realizable} and \eqref{eq: 02 Noe configurations not realizable} that
\begin{align}
\frac{1}{C_{21}} &= \alpha, \label{eq: 03 Noe configurations not realizable} \\
\frac{1}{R_{21} C_{21} C_{22}} &= \beta, \label{eq: 04 Noe configurations not realizable} \\
\frac{1}{L_{21} C_{21} C_{22}} &= \gamma,  \label{eq: 05 Noe configurations not realizable}  \\
\frac{C_{21} + C_{22}}{R_{21}C_{21}C_{22}} &= 2p + p_1, \label{eq: 06 Noe configurations not realizable} \\
\frac{1}{L_{21} C_{22}} &= p(p+2p_1), \label{eq: 07 Noe configurations not realizable} \\
\frac{1}{R_{21}L_{21}C_{21}C_{22}} &= p_1p^2. \label{eq: 08 Noe configurations not realizable}
\end{align}
It follows from  \eqref{eq: 03 Noe configurations not realizable} that
\begin{equation} \label{eq: C21 Noe configurations not realizable}
C_{21} = \frac{1}{\alpha}.
\end{equation}
By \eqref{eq: 04 Noe configurations not realizable}, \eqref{eq: 06 Noe configurations not realizable}, and \eqref{eq: C21 Noe configurations not realizable}, one obtains
\begin{equation} \label{eq: C22 Noe configurations not realizable}
C_{22} = \frac{2\alpha p + \alpha p_1 - \beta}{\alpha \beta}.
\end{equation}
By \eqref{eq: 04 Noe configurations not realizable}, \eqref{eq: C21 Noe configurations not realizable}, and \eqref{eq: C22 Noe configurations not realizable}, one obtains
\begin{equation} \label{eq: R21 Noe configurations not realizable}
R_{21} = \frac{\alpha^2}{2\alpha p + \alpha p_1 - \beta}.
\end{equation}
Then, \eqref{eq: 05 Noe configurations not realizable}, \eqref{eq: C21 Noe configurations not realizable}, and \eqref{eq: C22 Noe configurations not realizable} yield
\begin{equation} \label{eq: L21 Noe configurations not realizable}
L_{21} = \frac{\alpha^2 \beta}{\gamma (2\alpha p + \alpha p_1 - \beta)}.
\end{equation}
It follows from \eqref{eq: 07 Noe configurations not realizable}, \eqref{eq: C22 Noe configurations not realizable}, and \eqref{eq: L21 Noe configurations not realizable} that
\begin{equation} \label{eq: Condition01 Noe configurations not realizable}
\alpha p^2 + 2\alpha p_1 p - \gamma = 0.
\end{equation}
Moreover, substituting \eqref{eq: C21 Noe configurations not realizable}--\eqref{eq: L21 Noe configurations not realizable} into \eqref{eq: 08 Noe configurations not realizable} yields
\begin{equation} \label{eq: Condition02 Noe configurations not realizable}
\alpha^2 p_1 p^2 - 2\alpha \gamma p - \gamma (\alpha p_1 - \beta) = 0.
\end{equation}
The assumption that $R_{21} > 0$, $L_{21} > 0$, and $C_{22} > 0$ gives
\begin{equation} \label{eq: Condition03 Noe configurations not realizable}
2\alpha p + \alpha p_1 - \beta > 0.
\end{equation}
By \textcolor[rgb]{0.98,0.00,0.00}{(28)} and \eqref{eq: 02 Noe configurations not realizable}, $Z(s)$ is calculated as
\begin{equation} \label{eq: Z1 + Z2 Noe configurations not realizable}
Z(s) = Z_1(s) + Z_2(s) =
\frac{ms^3 + (2mp+\alpha + q)s^2 + (mp^2 + 2qp + \beta)s + (qp^2 + \gamma)}{(s+p_1)(s+p)^2}.
\end{equation}
Comparing (1) with \eqref{eq: Z1 + Z2 Noe configurations not realizable}, one obtains
\begin{align}
m  &= k, \label{eq: 101 Noe configurations not realizable}  \\
2mp + \alpha + q &= k(p_1 + 2z),  \label{eq: 102 Noe configurations not realizable}   \\
mp^2 + 2q p + \beta &= kz(z+2p_1),  \label{eq: 103 Noe configurations not realizable}   \\
q p^2 + \gamma &= kp_1z^2. \label{eq: 104 Noe configurations not realizable}
\end{align}
Then, \eqref{eq: 101 Noe configurations not realizable} and \eqref{eq: 102 Noe configurations not realizable} yield
\begin{equation}  \label{eq: alpha + q Noe configurations not realizable}
\alpha + q = k(p_1 + 2z - 2p).
\end{equation}
By \eqref{eq: Condition01 Noe configurations not realizable} and \eqref{eq: 104 Noe configurations not realizable}, one obtains
\begin{equation}  \label{eq: alpha + q 02 Noe configurations not realizable}
(p^2 + 2p_1p)\alpha + p^2 q = kp_1z^2.
\end{equation}
By \eqref{eq: alpha + q Noe configurations not realizable} and \eqref{eq: alpha + q 02 Noe configurations not realizable}, one obtains
\begin{equation} \label{eq: alpha Noe configurations not realizable}
\alpha = \frac{k(p-z)(2p^2-p_1p-p_1z)}{2p_1p}
\end{equation}
and
\begin{equation} \label{eq: q Noe configurations not realizable}
q = -\frac{k(2p^3+(3p_1-2z)p^2-(2p_1^2+4p_1z)p+p_1z^2)}{2p_1p}.
\end{equation}
Then, \eqref{eq: 101 Noe configurations not realizable}, \eqref{eq: 103 Noe configurations not realizable}, and \eqref{eq: q Noe configurations not realizable} yield
\begin{equation}  \label{eq: beta Noe configurations not realizable}
\beta = \frac{2k(p-z)(p^2+p_1p-p_1z-p_1^2)}{p_1}.
\end{equation}
It follows from \eqref{eq: 104 Noe configurations not realizable} and \eqref{eq: q Noe configurations not realizable} that
\begin{equation} \label{eq: gamma Noe configurations not realizable}
\gamma = \frac{k(p+2p_1)(p-z)(2p^2-p_1p-p_1z)}{2p_1}.
\end{equation}
Based on \eqref{eq: 101 Noe configurations not realizable} and \eqref{eq: q Noe configurations not realizable},  one implies that
\eqref{eq: 01 N1 00 Z2} is equivalent to
\begin{equation} \label{eq: equivalent Condition01 Noe configurations not realizable}
(p-z)(2p^2 + 3p_1p - p_1z) < 0,
\end{equation}
which implies $p < z$. Since the condition of Theorem~1 does not hold, one only needs to consider to the case of $p \leq z/(2 + \sqrt{5})$. Furthermore,  \eqref{eq: Condition03 Noe configurations not realizable} is equivalent to
\begin{equation} \label{eq: equivalent Condition02 Noe configurations not realizable}
(p-z)(4p^2 - (3p_1 + 2z)p + p_1z) < 0,
\end{equation}
by \eqref{eq: alpha Noe configurations not realizable} and \eqref{eq: beta Noe configurations not realizable}.
Combining \eqref{eq: equivalent Condition01 Noe configurations not realizable} and \eqref{eq: equivalent Condition02 Noe configurations not realizable}, one obtains
\begin{equation*}
\frac{2zp-4p^2}{z-3p} < p_1 < \frac{2p^2}{z-3p},
\end{equation*}
which implies $p > z/3$. Therefore, the condition of Theorem~1 holds.

Therefore, $Z(s) \in \mathcal{Z}_{p2,z2}$  cannot be realized as  in Fig.~2(a), where $N_1$ is the configuration in Fig.~\ref{fig: One-reactive Three-element N1}(a),
and $N_2$ is  the configurations in Fig.~\ref{fig: Three-reactive Four-element N2}(e).
\end{proof}

\begin{lemma} \label{lemma: four-reactive seven-element configurations lemma 1}
If a biquadratic impedance $Z(s) \in \mathcal{Z}_{p2,z2}$  is realizable as  in Fig.~2(a), where $N_1$ is the configuration in Fig.~\ref{fig: One-reactive Three-element N1}(a)
and $N_2$ is one of the configurations in Figs.~\ref{fig: Three-reactive Four-element N2}(b) and \ref{fig: Three-reactive Four-element N2}(f), then the condition of Lemma~1 holds.
\end{lemma}
\begin{proof}
It is calculated that the impedance of the configuration in Fig.~\ref{fig: Three-reactive Four-element N2}(b) is
\begin{equation} \label{eq: 01 Nob configurations not realizable}
Z_2(s) = \frac{R_{21} L_{21} L_{22} C_{21} s^3 + R_{21} L_{21} s}{L_{21} L_{22} C_{21} s^3 + R_{21} C_{21} (L_{21} + L_{22}) s^2 + L_{21} s + R_{21}}.
\end{equation}
If $Z(s)$ is realizable as  in Fig.~2(a), where $N_1$ is the configuration in Fig.~\ref{fig: One-reactive Three-element N1}(a) and $N_2$ is the configuration in Fig.~\ref{fig: Three-reactive Four-element N2}(b), then the impedance of $N_1$ is in the form of \eqref{eq: Z1 one-reactive three-element N1}, where $m$, $q$, $p_1$ $> 0$  and  \eqref{eq: Realizability of N1 one-reactive three-element}
 holds, and the impedance of $N_2$ is in the form of
\begin{equation}  \label{eq: 02 Nob configurations not realizable}
Z_2(s) = \frac{\alpha s^3 + \gamma s}{(s+p_1)(s+p)^2},
\end{equation}
where $\alpha$, $\gamma$ $> 0$.  Then, it follows from \eqref{eq: 01 Nob configurations not realizable} and \eqref{eq: 02 Nob configurations not realizable} that
\begin{align}
R_{21} &= \alpha, \label{eq: 03 Nob configurations not realizable} \\
\frac{R_{21}}{L_{22} C_{21}} &= \gamma, \label{eq: 04 Nob configurations not realizable} \\
\frac{R_{21} (L_{21} + L_{22})}{L_{21} L_{22}} &= 2p + p_1, \label{eq: 05 Nob configurations not realizable} \\
\frac{1}{L_{22} C_{21}} &= p(p+2p_1), \label{eq: 06 Nob configurations not realizable} \\
\frac{R_{21}}{L_{21} L_{22} C_{21}} &= p_1p^2. \label{eq: 07 Nob configurations not realizable}
\end{align}
By \eqref{eq: 03 Nob configurations not realizable}, \eqref{eq: 06 Nob configurations not realizable}, and \eqref{eq: 07 Nob configurations not realizable}, one obtains
\begin{equation}  \label{eq: L21 Nob configurations not realizable}
L_{21} = \frac{\alpha (p+2p_1)}{p_1p}.
\end{equation}
Then, \eqref{eq: 03 Nob configurations not realizable}, \eqref{eq: 05 Nob configurations not realizable}, and \eqref{eq: L21 Nob configurations not realizable} yield
\begin{equation}  \label{eq: L22 Nob configurations not realizable}
L_{22} = \frac{\alpha (p + 2p_1)}{2(p+p_1)^2}.
\end{equation}
By \eqref{eq: 06 Nob configurations not realizable} and \eqref{eq: L22 Nob configurations not realizable}, one obtains
\begin{equation}  \label{eq: C21 Nob configurations not realizable}
C_{21} = \frac{2(p+p_1)^2}{\alpha p (p + 2p_1)^2}.
\end{equation}
Then, \eqref{eq: 03 Nob configurations not realizable}, \eqref{eq: 04 Nob configurations not realizable}, \eqref{eq: L22 Nob configurations not realizable}, and \eqref{eq: C21 Nob configurations not realizable} yield
\begin{equation}  \label{eq: Condition01 Nob configurations not realizable}
\alpha p^2 + 2\alpha p_1 p - \gamma = 0.
\end{equation}
By \eqref{eq: Z1 one-reactive three-element N1} and \eqref{eq: 02 Nob configurations not realizable}, $Z(s)$ is calculated as
\begin{equation} \label{eq: Z1 + Z2 Nob configurations not realizable}
Z(s) = Z_1(s) + Z_2(s) = \frac{(m+\alpha) s^3 + (2mp + q) s^2 + (mp^2 + 2qp + \gamma) s + q p^2 }{(s+p_1)(s+p)^2}.
\end{equation}
Comparing (1) with \eqref{eq: Z1 + Z2 Nob configurations not realizable}, one obtains \begin{align}
m + \alpha &= k, \label{eq: 101 Nob configurations not realizable}  \\
2mp + q &= k(p_1 + 2z),  \label{eq: 102 Nob configurations not realizable}   \\
mp^2 + 2qp + \gamma &= kz(z+2p_1),  \label{eq: 103 Nob configurations not realizable}   \\
qp^2 &= kp_1z^2. \label{eq: 104 Nob configurations not realizable}
\end{align}
By \eqref{eq: 104 Nob configurations not realizable}, one obtains
\begin{equation}  \label{eq: q Nob configurations not realizable}
q = \frac{kp_1z^2}{p^2}.
\end{equation}
Then, \eqref{eq: 102 Nob configurations not realizable} and \eqref{eq: q Nob configurations not realizable} yield
\begin{equation}  \label{eq: m Nob configurations not realizable}
m = \frac{k(p_1p^2 + 2zp^2 - p_1z^2)}{2p^3}.
\end{equation}
By \eqref{eq: 101 Nob configurations not realizable} and \eqref{eq: m Nob configurations not realizable}, one obtains
\begin{equation} \label{eq: alpha Nob configurations not realizable}
\alpha = \frac{k(p-z)(2p^2 -  p_1 p - p_1 z)}{2p^3}.
\end{equation}
It follows from \eqref{eq: 103 Nob configurations not realizable}, \eqref{eq: q Nob configurations not realizable}, and \eqref{eq: m Nob configurations not realizable} that
\begin{equation}  \label{eq: gamma Nob configurations not realizable}
\gamma = \frac{-k(p-z)(p_1p +  2zp - 3zp_1)}{2p}.
\end{equation}
Together with \eqref{eq: q Nob configurations not realizable}--\eqref{eq: gamma Nob configurations not realizable}, condition~\eqref{eq: Realizability of N1 one-reactive three-element}
 is equivalent to
$p < z$, and \eqref{eq: Condition01 Nob configurations not realizable} and
$m > 0$ are equivalent to
\begin{equation}  \label{eq: Equivalent Condition01 Nob configurations not realizable}
(p+z)p_1^2 - 2p(p-z)p_1 - p^2(p+z) = 0
\end{equation}
and
\begin{equation} \label{eq: m positive Nob configurations not realizable}
p_1 < \frac{2zp^2}{z^2 - p^2},
\end{equation}
respectively.
Combining \eqref{eq: Equivalent Condition01 Nob configurations not realizable} and \eqref{eq: m positive Nob configurations not realizable}, one obtains
$z/(2 + \sqrt{3}) < p < z$, which satisfies the condition of Lemma~1.

Therefore, if a biquadratic impedance $Z(s) \in \mathcal{Z}_{p2,z2}$  is realizable as  in Fig.~2(a), where $N_1$ is the configuration in Fig.~\ref{fig: One-reactive Three-element N1}(a)
and $N_2$ is one of the configurations in Figs.~\ref{fig: Three-reactive Four-element N2}(b), then the condition of Lemma~1 holds.

It is calculated that the impedance of the configuration in Fig.~\ref{fig: Three-reactive Four-element N2}(f) is
\begin{equation} \label{eq: 01 Nof configurations not realizable}
Z_2(s) = \frac{R_{21} L_{21} C_{22} s^2 + L_{21} s + R_{21}}{R_{21} L_{21} C_{21} C_{22} s^3 + L_{21} C_{21} s^2 + R_{21} (C_{21} + C_{22}) s + 1}.
\end{equation}
If $Z(s)$ is realizable as  in Fig.~2(a), where $N_1$ is the configuration in Fig.~\ref{fig: One-reactive Three-element N1}(a) and $N_2$ is the configuration in Fig.~\ref{fig: Three-reactive Four-element N2}(f), then the impedance of $N_1$ is in the form of
\eqref{eq: Z1 one-reactive three-element N1}, where $m$, $q$, $p_1$ $> 0$  and \eqref{eq: Realizability of N1 one-reactive three-element} holds, and the impedance of $N_2$ is in the form of
\eqref{eq: 02 Noe configurations not realizable},
where $\alpha$, $\beta$, $\gamma$ $> 0$. Then, it follows from \eqref{eq: 02 Noe configurations not realizable} and \eqref{eq: 01 Nof configurations not realizable}    that
\begin{align}
\frac{1}{C_{21}} &= \alpha, \label{eq: 03 Nof configurations not realizable} \\
\frac{1}{R_{21} C_{21} C_{22}} &= \beta, \label{eq: 04 Nof configurations not realizable} \\
\frac{1}{L_{21} C_{21} C_{22}} &= \gamma,  \label{eq: 05 Nof configurations not realizable}  \\
\frac{1}{R_{21} C_{22}} &= 2p + p_1, \label{eq: 06 Nof configurations not realizable} \\
\frac{C_{21} + C_{22}}{L_{21} C_{21} C_{22}} &= p(p+2p_1), \label{eq: 07 Nof configurations not realizable} \\
\frac{1}{R_{21} L_{21} C_{21} C_{22}} &= p_1p^2. \label{eq: 08 Nof configurations not realizable}
\end{align}
It follows from \eqref{eq: 03 Nof configurations not realizable} that
\begin{equation}  \label{eq: C21 Nof configurations not realizable}
C_{21} = \frac{1}{\alpha}.
\end{equation}
Then, \eqref{eq: 05 Nof configurations not realizable}, \eqref{eq: 07 Nof configurations not realizable}, and \eqref{eq: C21 Nof configurations not realizable} yield
\begin{equation}  \label{eq: C22 Nof configurations not realizable}
C_{22} = \frac{\alpha p^2 + 2\alpha p_1 p - \gamma}{\alpha \gamma}.
\end{equation}
By \eqref{eq: 06 Nof configurations not realizable} and \eqref{eq: C22 Nof configurations not realizable}, one obtains
\begin{equation}  \label{eq: R21 Nof configurations not realizable}
R_{21} = \frac{\alpha \gamma}{(\alpha p^2 + 2\alpha p_1 p - \gamma)(2 p + p_1)}.
\end{equation}
Then, it follows from \eqref{eq: 05 Nof configurations not realizable}, \eqref{eq: C21 Nof configurations not realizable}, and \eqref{eq: C22 Nof configurations not realizable} that
\begin{equation}  \label{eq: L21 Nof configurations not realizable}
L_{21} = \frac{\alpha^2}{\alpha p^2 + 2\alpha p_1 p - \gamma}.
\end{equation}
By \eqref{eq: 04 Nof configurations not realizable} and \eqref{eq: C21 Nof configurations not realizable}--\eqref{eq: R21 Nof configurations not realizable}, one obtains
\begin{equation} \label{eq: Condition01 Nof configurations not realizable}
2 \alpha p - \beta + \alpha p_1 = 0.
\end{equation}
The assumption that $R_{21} > 0$, $L_{21} > 0$, and $C_{22} > 0$ gives
\begin{equation} \label{eq: Condition02 Nof configurations not realizable}
\alpha p^2 + 2\alpha p_1 p - \gamma > 0.
\end{equation}
Then, \eqref{eq: 08 Nof configurations not realizable}--\eqref{eq: L21 Nof configurations not realizable} yield
\begin{equation} \label{eq: Condition03 Nof configurations not realizable}
2\alpha p p_1^2 + (4\alpha p^2 - \gamma) p_1 + 2p(\alpha p^2 - \gamma) = 0.
\end{equation}
By (28) and \eqref{eq: 02 Noe configurations not realizable}, $Z(s)$ is calculated as
\begin{equation} \label{eq: Z1 + Z2 Nof configurations not realizable}
Z(s) = Z_1(s) + Z_2(s) =
\frac{ms^3 + (2mp + \alpha + q)s^2 + (mp^2 + 2qp + \beta)s + (qp^2 + \gamma)}{(s+p_1)(s+p)^2}.
\end{equation}
Comparing (1) with \eqref{eq: Z1 + Z2 Nof configurations not realizable}, one obtains
\begin{align}
m  &= k, \label{eq: 101 Nof configurations not realizable}  \\
2mp + \alpha + q &= k(p_1 + 2z),  \label{eq: 102 Nof configurations not realizable}   \\
mp^2 + 2q p + \beta &= kz(z+2p_1),  \label{eq: 103 Nof configurations not realizable}   \\
q p^2 + \gamma &= kp_1z^2. \label{eq: 104 Nof configurations not realizable}
\end{align}
Then, it follows from \eqref{eq: Condition01 Nof configurations not realizable} and \eqref{eq: 101 Nof configurations not realizable}--\eqref{eq: 103 Nof configurations not realizable} that
\begin{align}
\alpha &= \frac{k(p-z)(3p-z-2p_1)}{p_1},
\label{eq: alpha Nof configurations not realizable}   \\
q &= -\frac{k(3p^2 - 4zp - p_1^2 + z^2)}{p_1}.
\label{eq: q Nof configurations not realizable}
\end{align}
By \eqref{eq: 104 Nof configurations not realizable} and \eqref{eq: q Nof configurations not realizable}, one obtains
\begin{equation} \label{eq: gamma Nof configurations not realizable}
\gamma = \frac{k(p-z)((3p-z)p^2 - (p+z)p_1^2)}{p_1}.
\end{equation}
Then, \eqref{eq: Condition01 Nof configurations not realizable} and \eqref{eq: alpha Nof configurations not realizable} yield
\begin{equation} \label{eq: beta Nof configurations not realizable}
\beta = \frac{k(p-z)(3p-z-2p_1)(2p+p_1)}{p_1}.
\end{equation}
Together with \eqref{eq: 101 Nof configurations not realizable}  and \eqref{eq: q Nof configurations not realizable}, condition~\eqref{eq: Realizability of N1 one-reactive three-element} is equivalent to $z/3 < p < z$,  which satisfies the condition of Lemma~1.

Therefore, if a biquadratic impedance $Z(s) \in \mathcal{Z}_{p2,z2}$  is realizable as  in Fig.~2(a), where $N_1$ is the configuration in Fig.~\ref{fig: One-reactive Three-element N1}(a)
and $N_2$ is one of the configurations in Figs.~\ref{fig: Three-reactive Four-element N2}(f), then the condition of Lemma~1 holds.
\end{proof}

\begin{lemma}   \label{lemma: configuration Nog realizable}
A biquadratic impedance $Z(s) \in \mathcal{Z}_{p2,z2}$ can be realized as the configuration in Fig.~2(a), where $N_1$ is the configuration in Fig.~\ref{fig: One-reactive Three-element N1}(a)
and $N_2$ is the configurations in Figs.~\ref{fig: Three-reactive Four-element N2}(g) (that is, the configuration in Fig.~3(a)), if and only if
\begin{align}
(p-z)(p-3z) &> 0,
\label{lemma: condition01 configuration Nog realizable}  \\
p^4 - 6zp^3 + 6z^2p^2 - 14z^3p + 5z^4 &< 0.
\label{lemma: condition configuration Nog realizable}
\end{align}
\end{lemma}
\begin{proof}
The realizability condition of a general biquadratic impedance $Z(s)$ in the form of \eqref{eq: general biquadratic impedances} as a configuration that is equivalent to Fig.~3(a) is available in \cite[Table~I]{JZ14}. Letting $A = kx$, $B = 2kzx$, $C = k z^2x$, $D = x$, $E = 2px$, and $F = p^2x$ for  $x > 0$, the realizability condition for such a specific biquadratic impedance can be derived.

The element values can be derived as $R_1 = q/p_1$, $R_2 = mq/(q-mp_1)$, $C_1 = (q - mp_1)/q^2$, $R_{21} = \alpha$, $L_{21} = \alpha/(2p + p_1)$, $L_{22} = \alpha \beta/\gamma$, and $C_{21} = 1/\beta$, where $\alpha = k(p-z)(2p+p_1)(p^2+zp-2zp_1)/(2p^4)$, $\beta = (2k(p-z)(-zp_1^2+p(p-z)p_1+zp^2))/p^3$, $\gamma = (kp_1(p-z)(p^2+zp-2zp_1))/(2p^2)$, and $p_1$ is a positive root of $(3z-p)p_1^2 - 2p(p-z)p_1 + p^2(p-3z) = 0$.
\end{proof}

\begin{lemma}  \label{lemma: configuration Noh realizable}
If a biquadratic impedance $Z(s) \in \mathcal{Z}_{p2,z2}$ can be realized as  in Fig.~2(a), where $N_1$ is the configuration in Fig.~\ref{fig: One-reactive Three-element N1}(a) and $N_2$ is the configuration in Fig.~\ref{fig: Three-reactive Four-element N2}(h), then the condition of Lemma~\ref{lemma: configuration Nog realizable} holds.
\end{lemma}
\begin{proof}
By calculation,
the impedance of the configuration in
Fig.~\ref{fig: Three-reactive Four-element N2}(h) is obtained as
\begin{equation} \label{eq: 01 Noh configurations not realizable}
Z_2(s) =  \frac{s(R_{21}L_{21}L_{22}C_{21} s^2 + L_{21} L_{22} s + R_{21} L_{21})}{L_{21} L_{22} C_{21} s^3 + R_{21} (L_{21} + L_{22})C_{21} s^2 + L_{22} s + R_{21}}.
\end{equation}
If $Z(s)$ is realizable as in Fig.~2(a), where $N_1$ is the configuration in Fig.~\ref{fig: One-reactive Three-element N1}(a) and $N_2$ is the configuration in Fig.~\ref{fig: Three-reactive Four-element N2}(h), then the impedance of $N_1$ is in the form of \eqref{eq: Z1 one-reactive three-element N1}, where $m$, $q$, $p_1$ $> 0$  and \eqref{eq: Realizability of N1 one-reactive three-element} holds, and moreover the impedance of $N_2$ is in the form of
\begin{equation}  \label{eq: 02 Noh configurations not realizable}
Z_2(s) = \frac{s(\alpha s^2 + \beta s + \gamma)}{(s+p_1)(s+p)^2},
\end{equation}
where $\alpha$, $\beta$, $\gamma$ $> 0$.  Consequently, it follows from \eqref{eq: 01 Noh configurations not realizable} and
\eqref{eq: 02 Noh configurations not realizable} that
\begin{align}
R_{21} &= \alpha, \label{eq: 03 Noh configurations not realizable} \\
\frac{1}{C_{21}} &= \beta, \label{eq: 04 Noh configurations not realizable} \\
\frac{R_{21}}{L_{22} C_{21}} &= \gamma,  \label{eq: 05 Noh configurations not realizable}  \\
\frac{R_{21} (L_{21} + L_{22})}{L_{21} L_{22}} &= 2p + p_1, \label{eq: 06 Noh configurations not realizable} \\
\frac{1}{L_{21} C_{21}} &= p(p+2p_1), \label{eq: 07 Noh configurations not realizable} \\
\frac{R_{21}}{L_{21} L_{22} C_{21}} &= p_1p^2. \label{eq: 08 Noh configurations not realizable}
\end{align}
Thus, \eqref{eq: 04 Noh configurations not realizable} yields
\begin{equation}  \label{eq: C21 Noh configurations not realizable}
C_{21} = \frac{1}{\beta}.
\end{equation}
By \eqref{eq: 05 Noh configurations not realizable} and \eqref{eq: 08 Noh configurations not realizable}, one obtains
\begin{equation}  \label{eq: L21 Noh configurations not realizable}
L_{21} = \frac{\gamma}{p_1p^2}.
\end{equation}
It follows from \eqref{eq: 03 Noh configurations not realizable}, \eqref{eq: 05 Noh configurations not realizable}, and \eqref{eq: C21 Noh configurations not realizable} that
\begin{equation}  \label{eq: L22 Noh configurations not realizable}
L_{22} = \frac{\alpha \beta}{\gamma}.
\end{equation}
By \eqref{eq: 03 Noh configurations not realizable}, \eqref{eq: 06 Noh configurations not realizable}, \eqref{eq: L21 Noh configurations not realizable}, and \eqref{eq: L22 Noh configurations not realizable}, one obtains
\begin{equation} \label{eq: Condition01 Noh configurations not realizable}
(\alpha p^2 - \gamma)\beta p_1 + \gamma(\gamma - 2\beta p) = 0.
\end{equation}
It follows from \eqref{eq: 07 Noh configurations not realizable}, \eqref{eq: C21 Noh configurations not realizable}, and \eqref{eq: L21 Noh configurations not realizable} that
\begin{equation} \label{eq: Condition02 Noh configurations not realizable}
(\beta p - 2\gamma)p_1 - \gamma p = 0.
\end{equation}
Based on \eqref{eq: Z1 one-reactive three-element N1} and \eqref{eq: 02 Noh configurations not realizable}, calculation yields
\begin{equation} \label{eq: Z1 + Z2 Noh configurations not realizable}
\begin{split}
&Z(s) = Z_1(s) + Z_2(s) =  \\
&\frac{(m+\alpha)s^3 + (2mp+\beta+q)s^2 + (mp^2 + 2qp + \gamma)s + qp^2}{s^3 + (2p+p_1)s^2 + p(p+2p_1)s + p_1p^2}.
\end{split}
\end{equation}
Comparing (1) with \eqref{eq: Z1 + Z2 Noh configurations not realizable}, one obtains
\begin{align}
m + \alpha  &= k, \label{eq: 101 Noh configurations not realizable}  \\
2mp + \beta + q &= k(p_1 + 2z),  \label{eq: 102 Noh configurations not realizable}   \\
mp^2 + 2qp + \gamma &= kz(z+2p_1),  \label{eq: 103 Noh configurations not realizable}   \\
q p^2 &= kp_1z^2. \label{eq: 104 Noh configurations not realizable}
\end{align}
It follows from \eqref{eq: 104 Noh configurations not realizable} that \begin{equation}  \label{eq: q Noh configurations not realizable}
q = \frac{kz^2p_1}{p^2}.
\end{equation}
Thus, \eqref{eq: Condition02 Noh configurations not realizable}, \eqref{eq: 102 Noh configurations not realizable}, \eqref{eq: 103 Noh configurations not realizable}, and \eqref{eq: q Noh configurations not realizable} together yield
\begin{equation}  \label{eq: beta Noh configurations not realizable}
\beta = \frac{k(p-z)(p+2p_1)((p-3z)p_1+2zp)}{p^3},
\end{equation}
and
\begin{equation}  \label{eq: m Noh configurations not realizable}
m = \frac{-k((p-z)(p-3z)p_1^2-z^2p^2)}{p^4}.
\end{equation}
By \eqref{eq: 101 Noh configurations not realizable} and
\eqref{eq: m Noh configurations not realizable}, one obtains
\begin{equation}  \label{eq: alpha Noh configurations not realizable}
\alpha = \frac{k(p-z)((p-3z)p_1^2+p^2(p+z))}{p^4}.
\end{equation}
It follows from \eqref{eq: Condition02 Noh configurations not realizable} and \eqref{eq: beta Noh configurations not realizable} that
\begin{equation}  \label{eq: gamma Noh configurations not realizable}
\gamma = \frac{kp_1(p-z)((p-3z)p_1+2zp)}{p^2}.
\end{equation}
Together with \eqref{eq: q Noh configurations not realizable} and \eqref{eq: m Noh configurations not realizable},
condition~\eqref{eq: Realizability of N1 one-reactive three-element} is equivalent to \eqref{lemma: condition01 configuration Nog realizable}. Together with \eqref{eq: beta Noh configurations not realizable}, \eqref{eq: alpha Noh configurations not realizable}, and \eqref{eq: gamma Noh configurations not realizable},
condition~\eqref{eq: Condition01 Noh configurations not realizable} is equivalent to
\begin{equation} \label{eq: equivalent Condition01 Noh configurations not realizable}
(5z-3p)p_1^2 + (p-3z)p^2 = 0.
\end{equation}
From \eqref{eq: m Noh configurations not realizable}, it follows that $m > 0$ is equivalent to $p_1^2 < z^2p^2/((p-z)(p-3z))$, which, in turn, is   equivalent to $p^2 - 5zp + 2z^2 < 0$, together with \eqref{lemma: condition01 configuration Nog realizable} and \eqref{eq: equivalent Condition01 Noh configurations not realizable}. Therefore, the condition of Lemma~\ref{lemma: configuration Nog realizable} must hold.
\end{proof}

\section{Supplementary Lemmas of Five-Reactive Seven-Element Series-Parallel Realizations for the Proof of  Lemma~4}

\begin{lemma}
\label{lemma: Possible configurations of N1 and N2 two-reactive three-element and N2 three-reactive four-element}
{Consider the five-reactive  seven-element series-parallel network in  Fig.~2(a),  realizing a biquadratic impedance $Z(s)$ in the form of \eqref{eq: general biquadratic impedances} with $A$, $B$, $C$, $D$, $E$, $F$ $> 0$, where $N_1$ is a three-element series-parallel network and $N_2$ is a  four-element series-parallel network. If $Z(s)$ cannot be realized as a series-parallel network containing fewer than seven elements,  then $N_1$ is one of the configurations in Figs.~\ref{fig: Two-reactive configurations N1}(b)--\ref{fig: Two-reactive configurations N1}(d) and $N_2$ will be equivalent to one of the configurations in Fig.~\ref{fig: Three-reactive Four-element N2}.}
\end{lemma}
\begin{proof}
By \cite[Lemma~2]{WCH14}, $Z(s)$ cannot be realized as the series connection of two networks, one of which contains only reactive elements.
Therefore, $N_1$ can only contain two reactive elements.
For any realization of $Z(s)$,
there is no cut-set $\mathcal{C}(a,a')$ corresponding to one kind of reactive elements, where $a$ and $a'$ denote two terminals,
by \cite[Lemma~1]{WCH14}.
The possible network graphs for subnetworks $N_1$ and $N_2$ are listed  in Figs.~\ref{fig: Three-element graphs} and \ref{fig: Four-element graphs}, respectively. Based on the method of enumeration and the equivalence in \cite[Lemma~11]{JS11}, $N_1$ is one of the configurations in Figs.~\ref{fig: Two-reactive configurations N1}(b)--\ref{fig: Two-reactive configurations N1}(d) and $N_2$ can be equivalent to one of the configurations in Fig.~\ref{fig: Three-reactive Four-element N2}.
\end{proof}

\begin{lemma} \label{lemma: configurations not realizable}
A biquadratic impedance $Z(s) \in \mathcal{Z}_{p2,z2}$
not satisfying the condition of Lemma~3
cannot be realized as  in Fig.~2(a), where $N_1$ is the configuration in Fig.~\ref{fig: Two-reactive configurations N1}(b) and $N_2$ is one of the configurations in Figs.~\ref{fig: Three-reactive Four-element N2}(a)--\ref{fig: Three-reactive Four-element N2}(d), \ref{fig: Three-reactive Four-element N2}(f), and  \ref{fig: Three-reactive Four-element N2}(h).
\end{lemma}
\begin{proof}
It has been shown that the impedance of the configuration in Fig.~\ref{fig: Two-reactive configurations N1}(b) is in the form of
\begin{equation} \label{eq: Z1 b configuration}
Z_1(s) = \frac{R_1L_1C_1 s^2 + R_1}{L_1C_1 s^2 + R_1C_1 s + 1},
\end{equation}
and the impedance of the configuration in Fig.~\ref{fig: Three-reactive Four-element N2}(a) is in the form of \eqref{eq: 01 Noa configurations not realizable}. Since it is assumed that the condition of Lemma~3
does not hold, $Z(s) \in \mathcal{Z}_{p2,z2}$ cannot be realized with fewer than five reactive elements.
If $Z(s)$ is realizable as  in Fig.~2(a), where $N_1$ is the configuration in Fig.~\ref{fig: Two-reactive configurations N1}(b) and $N_2$ is the configuration in Fig.~\ref{fig: Three-reactive Four-element N2}(a), then the impedance of $N_1$ is of degree two and is in the form of
\begin{equation} \label{eq: Z1 No2 (s+p1)(s+p)}
Z_1(s) = \frac{m(s^2 + pp_1)}{(s+p_1)(s+p)},
\end{equation}
where $m$, $p_1$, $p$ $> 0$, and the impedance of $N_2$ is of degree three and is in the form of
\eqref{eq: 02 Noa configurations not realizable}, where $\alpha$, $\gamma$ $> 0$, and \eqref{eq: Condition01 Noa configurations not realizable} and \eqref{eq: Condition02 Noa configurations not realizable} hold. Furthermore,
\begin{equation}  \label{eq: Z1 + Z2 No2 and Noa configurations}
\begin{split}
&Z(s) = Z_1(s) + Z_2(s) =\\
& \frac{m s^3 + (m p + \alpha)s^2 + m p_1 p s + (mp_1p^2 + \gamma)}{(s+p_1)(s+p)^2}.
\end{split}
\end{equation}
Comparing (1) with \eqref{eq: Z1 + Z2 No2 and Noa configurations}, one obtains
\begin{align}
m &= k,     \label{eq: 101 Z1 + Z2 No2 and Noa configurations} \\
mp + \alpha &= k(p_1 + 2z),  \label{eq: 102 Z1 + Z2 No2 and Noa configurations}  \\
mp_1p &= kz(2p_1+z),  \label{eq: 103 Z1 + Z2 No2 and Noa configurations}  \\
mp_1p^2 + \gamma &= kz^2p_1.  \label{eq: 104 Z1 + Z2 No2 and Noa configurations}
\end{align}
By \eqref{eq: 101 Z1 + Z2 No2 and Noa configurations}--\eqref{eq: 103 Z1 + Z2 No2 and Noa configurations},  one obtains
\begin{equation}  \label{eq: p_1 Z1 + Z2 No2 and Noa configurations}
p_1 = \frac{z^2}{p-2z}.
\end{equation}
By \eqref{eq: 101 Z1 + Z2 No2 and Noa configurations}, \eqref{eq: 104 Z1 + Z2 No2 and Noa configurations}, and \eqref{eq: p_1 Z1 + Z2 No2 and Noa configurations}, one obtains
\begin{equation}   \label{eq: gamma Z1 + Z2 No2 and Noa configurations}
\gamma = \frac{-kz^2(p-z)(p+z)}{p-2z}.
\end{equation}
It follows from \eqref{eq: p_1 Z1 + Z2 No2 and Noa configurations} and $p_1 > 0$ that
$p > 2z$, which further implies that $\gamma < 0$ by \eqref{eq: gamma Z1 + Z2 No2 and Noa configurations}. This is impossible.

Therefore, $Z(s) \in \mathcal{Z}_{p2,z2}$  cannot be realized as  in Fig.~2(a), where $N_1$ is the configuration in
Fig.~\ref{fig: Two-reactive configurations N1}(b) and $N_2$ is  the configuration in Fig.~\ref{fig: Three-reactive Four-element N2}(a).

It is clear that any network in Fig.~2(a), where $N_1$ is the configuration in
Fig.~\ref{fig: Two-reactive configurations N1}(b) and $N_2$ is  the configuration in Fig.~\ref{fig: Three-reactive Four-element N2}(b), can be a frequency inverse dual network of another one in Fig.~2(a), where $N_1$ is the configuration in
Fig.~\ref{fig: Two-reactive configurations N1}(b) and $N_2$ is  the configuration in Fig.~\ref{fig: Three-reactive Four-element N2}(a). Based on the principle of frequency inverse, $Z(s) \in \mathcal{Z}_{p2,z2}$  cannot be realized as such a network.

It has been shown that the impedance of the configuration in Fig.~\ref{fig: Two-reactive configurations N1}(b) is in the form of \eqref{eq: Z1 b configuration}, and the impedance of the configuration in Fig.~\ref{fig: Three-reactive Four-element N2}(c) is in the form of $Z(s) = (R_{21}L_{21}C_{22}s^2 + L_{21}s)/(R_{21}L_{21}C_{21}C_{22}s^3 + L_{21}(C_{21}+C_{22})s^2 + R_{21}C_{22}s + 1)$.
If $Z(s)$ is realizable as  in Fig.~2(a), where $N_1$ is the configuration in Fig.~\ref{fig: Two-reactive configurations N1}(b) and $N_2$ is the configuration in Fig.~\ref{fig: Three-reactive Four-element N2}(c), then the impedance of $N_1$ is of degree two and  is in the form of
\eqref{eq: Z1 No2 (s+p1)(s+p)},
where $m$, $p_1$, $p$ $> 0$, and the impedance of $N_2$ is of degree three and is in the form of
$Z(s) = (\alpha s^2 + \beta s)/((s+p_1)(s+p)^2)$, where $\alpha$, $\beta$ $> 0$. Furthermore,
\begin{equation}  \label{eq: Z1 + Z2 No2 and Noc configurations}
\begin{split}
&Z(s)
= Z_1(s) + Z_2(s)= \\
& \frac{ms^3 + (mp+\alpha)s^2 + (mp_1p+\beta)s + mp_1p^2}{(s+p_1)(s+p)^2}.
\end{split}
\end{equation}
Comparing (1) with \eqref{eq: Z1 + Z2 No2 and Noc configurations}, one obtains $m = k$ and $mp_1p^2 = kz^2p_1$, which further implies $p=z$. This contradicts the assumption.

Therefore, $Z(s) \in \mathcal{Z}_{p2,z2}$  cannot be realized as  in Fig.~2(a), where $N_1$ is the configuration in
Fig.~\ref{fig: Two-reactive configurations N1}(b) and $N_2$ is  the configuration in Fig.~\ref{fig: Three-reactive Four-element N2}(c).

It is clear that any network in Fig.~2(a), where $N_1$ is the configuration in
Fig.~\ref{fig: Two-reactive configurations N1}(b) and $N_2$ is  the configuration in Fig.~\ref{fig: Three-reactive Four-element N2}(d), can be a frequency inverse dual network of another one in Fig.~2(a), where $N_1$ is the configuration in
Fig.~\ref{fig: Two-reactive configurations N1}(b) and $N_2$ is  the configuration in Fig.~\ref{fig: Three-reactive Four-element N2}(c). Based on the principle of frequency inverse, $Z(s) \in \mathcal{Z}_{p2,z2}$  cannot be realized as such a network.

It has been shown that the impedance of the configuration in Fig.~\ref{fig: Two-reactive configurations N1}(b) is in the form of
\eqref{eq: Z1 b configuration}, and the impedance of the configuration in Fig.~8(f) is in the form of
\eqref{eq: 01 Nof configurations not realizable}.
If $Z(s)$ is realizable as  in Fig.~2(a), where $N_1$ is the configuration in Fig.~\ref{fig: Two-reactive configurations N1}(b) and $N_2$ is the configuration in Fig.~8(f), then the impedance of $N_1$ is of degree two and is in the form of
\eqref{eq: Z1 No2 (s+p1)(s+p)},
where $m$, $p_1$, $p$ $> 0$ and the impedance of $N_2$ is degree three and is in the form of \eqref{eq: 02 Noe configurations not realizable}, where $\alpha$, $\beta$, $\gamma$ $> 0$ and \eqref{eq: Condition01 Nof configurations not realizable}--\eqref{eq: Condition03 Nof configurations not realizable} hold.
Furthermore, one obtains
\eqref{eq: Z1 + Z2 No2 and Noe configurations}--\eqref{eq: gamma No2 and Noe configurations}.
Substituting \eqref{eq: 101 No2 and Noe configurations} and \eqref{eq: alpha No2 and Noe configurations}--\eqref{eq: gamma No2 and Noe configurations} into \eqref{eq: Condition01 Nof configurations not realizable}--\eqref{eq: Condition03 Nof configurations not realizable} yields
\begin{align}
p_1^2 + 2pp_1 - (2p^2-4zp+z^2) = 0,
\label{eq: equivalent Condition01 Nof configurations not realizable} \\
2pp_1^2 + z(4p-z)p_1 - p^2(p-2z) > 0,
\label{eq: equivalent Condition02 Nof configurations not realizable} \\
2pp_1^3 + (3p^2+4zp-z^2)p_1^2 + 2zp(4p-z)p_1 - 2p^3(p-2z) = 0,
\label{eq: equivalent Condition03 Nof configurations not realizable}
\end{align}
respectively.  By \eqref{eq: equivalent Condition01 Nof configurations not realizable}, \eqref{eq: equivalent Condition03 Nof configurations not realizable} can be further equivalent to
\begin{equation}  \label{eq: equivalent02 Condition03 Nof configurations not realizable}
(p^2 - 4zp + z^2)p_1^2 - 4p^3p_1 + 2p^3(p-2z) = 0.
\end{equation}
It is calculated that the resultant of \eqref{eq: equivalent Condition01 Nof configurations not realizable} and \eqref{eq: equivalent02 Condition03 Nof configurations not realizable} in $p_1$ is $8p^8 + 48zp^7 - 312z^2p^6 + 624z^3p^5 - 617z^4p^4 + 336z^5p^3 - 102z^6p^2 + 16z^7p - z^8$.
Since the condition of Lemma~3 does not hold, there exists at least one common root between \eqref{eq: equivalent Condition01 Nof configurations not realizable} and \eqref{eq: equivalent02 Condition03 Nof configurations not realizable} in $p_1$ if and only if
\begin{equation} \label{eq: Condition five-reactive-element configurations realizable 02}
\begin{split}
8p^8 + 48zp^7 &- 312z^2p^6 + 624z^3p^5 - 617z^4p^4  \\
&+ 336z^5p^3- 102z^6p^2 + 16z^7p - z^8 = 0
\end{split}
\end{equation}
holds with $p < z/(2 + \sqrt{5})$, which implies that the condition of Lemma~3 holds.
One further implies    $p_1 > 0$.
By \eqref{eq: equivalent Condition01 Nof configurations not realizable}, it is implied that \eqref{eq: equivalent Condition02 Nof configurations not realizable} is equivalent to
\begin{equation}  \label{eq: equivalent02 Condition02 Nof configurations not realizable}
p_1 < \frac{p(3p^2 - 6zp + 2z^2)}{(2p-z)^2}.
\end{equation}
By \eqref{eq: equivalent02 Condition03 Nof configurations not realizable},  \eqref{eq: equivalent02 Condition02 Nof configurations not realizable} is further equivalent to $(p-z)(7p^5 + 63zp^4 - 174z^2p^3 + 134z^3p^2 - 40z^4p + 4z^5) < 0$, which is verified to be satisfied.

It is clear that any network in Fig.~2(a), where $N_1$ is the configuration in
Fig.~\ref{fig: Two-reactive configurations N1}(b) and $N_2$ is  the configuration in Fig.~\ref{fig: Three-reactive Four-element N2}(h), can be a frequency inverse dual network of another one in Fig.~2(a), where $N_1$ is the configuration in
Fig.~\ref{fig: Two-reactive configurations N1}(b) and $N_2$ is  the configuration in Fig.~\ref{fig: Three-reactive Four-element N2}(f). Based on the principle of frequency inverse, $Z(s) \in \mathcal{Z}_{p2,z2}$  cannot be realized as such a network either.
\end{proof}

\begin{lemma} \label{lemma: five-reactive-element configurations realizable 01}
A biquadratic impedance $Z(s) \in \mathcal{Z}_{p2,z2}$ not satisfying the condition of Lemma~3 is realizable as  in Fig.~2(a), where $N_1$ is the configuration in Fig.~\ref{fig: Two-reactive configurations N1}(b) and $N_2$ is the configuration in Fig.~\ref{fig: Three-reactive Four-element N2}(e) (that is, the configuration in Fig.~4(a) whose one-terminal-pair labeled graph is $N_{4a}$), if and only if
\begin{equation} \label{eq: Condition five-reactive-element configurations realizable 01}
16 p^4 - 40 z p^3 + 31 z^2 p^2 - 10 z^3 p + z^4 = 0
\end{equation}
and $p < z/(2 + \sqrt{5})$
(it can be verified  that there is only one distinct root of the equation $16 \eta^4 - 40 \eta^3 + 31 \eta^2 - 10 \eta + 1 = 0$ for $\eta \in (0, 1/(2 + \sqrt{5}))$).
\end{lemma}
\begin{proof}
\textit{Necessity.}
It has been shown that the impedance of the configuration in Fig.~\ref{fig: Two-reactive configurations N1}(b)
is in the form of
\eqref{eq: Z1 b configuration}, and the impedance of the configuration in Fig.~\ref{fig: Three-reactive Four-element N2}(e) is in the form of
\eqref{eq: 01 Noe configurations not realizable}.
Since it is assumed that the condition of Lemma~3 does not hold, $Z(s) \in \mathcal{Z}_{p2,z2}$ cannot be realized with fewer than five reactive elements.
If $Z(s)$ is realizable as in Fig.~2(a), where $N_1$ is the configuration in Fig.~\ref{fig: Two-reactive configurations N1}(b) and $N_2$ is the configuration in Fig.~\ref{fig: Three-reactive Four-element N2}(e), then the impedance of $N_1$ is of degree two and is in the form of
\eqref{eq: Z1 No2 (s+p1)(s+p)},
where $m$, $p_1$, $p$ $> 0$, and the impedance of $N_2$ is of degree three and  is in the form of \eqref{eq: 02 Noe configurations not realizable}, where $\alpha$, $\beta$ $> 0$ and \eqref{eq: Condition01 Noe configurations not realizable}--\eqref{eq: Condition03 Noe configurations not realizable} hold.
Furthermore, one obtains
\begin{equation}  \label{eq: Z1 + Z2 No2 and Noe configurations}
Z(s) = Z_1(s) + Z_2(s) =
\frac{ms^3 + (mp+\alpha)s^2 + (mp_1p+\beta)s + (mp_1p^2 + \gamma)}{(s+p_1)(s+p)^2}.
\end{equation}
Comparing (1) with \eqref{eq: Z1 + Z2 No2 and Noe configurations}, one obtains
\begin{align}
m &= k,
\label{eq: 101 No2 and Noe configurations} \\
mp + \alpha &= k(p_1 + 2z),
\label{eq: 102 No2 and Noe configurations}  \\
mp_1p + \beta &= kz(2p_1+z),
\label{eq: 103 No2 and Noe configurations}  \\
mp_1p^2 + \gamma &= kz^2p_1.
\label{eq: 104 No2 and Noe configurations}
\end{align}
Then, \eqref{eq: 101 No2 and Noe configurations} and \eqref{eq: 102 No2 and Noe configurations} yield
\begin{equation} \label{eq: alpha No2 and Noe configurations}
\alpha = -k(p-p_1-2z).
\end{equation}
By \eqref{eq: 101 No2 and Noe configurations} and \eqref{eq: 103 No2 and Noe configurations}, one obtains
\begin{equation} \label{eq: beta No2 and Noe configurations}
\beta = -k(p_1p - 2zp_1 - z^2).
\end{equation}
It follows from \eqref{eq: 101 No2 and Noe configurations} and \eqref{eq: 104 No2 and Noe configurations} that
\begin{equation} \label{eq: gamma No2 and Noe configurations}
\gamma = -kp_1(p-z)(p+z).
\end{equation}
Substituting \eqref{eq: 101 No2 and Noe configurations} and
\eqref{eq: alpha No2 and Noe configurations}--\eqref{eq: gamma No2 and Noe configurations} into \eqref{eq: Condition01 Noe configurations not realizable}--\eqref{eq: Condition03 Noe configurations not realizable} yields
\begin{equation} \label{eq: equivalent Condition01 No2 and Noe configurations not realizable}
2pp_1^2 + z(4p-z)p_1 - p^2(p-2z) = 0,
\end{equation}
\begin{equation} \label{eq: equivalent Condition02 No2 and Noe configurations not realizable}
\begin{split}
(2p^2 - z^2)p_1^2 &+ 2pz(2p-z)p_1 \\
&- (p^4 -5z^2p^2 + 4z^3p - z^4) = 0,
\end{split}
\end{equation}
\begin{equation} \label{eq: equivalent Condition03 No2 and Noe configurations not realizable}
p_1^2 + 2pp_1 - (2p^2 - 4zp + z^2) > 0,
\end{equation}
respectively. It is calculated that the resultant of \eqref{eq: equivalent Condition01 No2 and Noe configurations not realizable} and \eqref{eq: equivalent Condition02 No2 and Noe configurations not realizable} in $p_1$ is $z^2(p+z)(p-z)^3(16p^4 - 40zp^3 + 31z^2p^2 - 10z^3p + z^4)$.
Since the condition of Lemma~3 does not hold, there exists at least one common root between \eqref{eq: equivalent Condition01 No2 and Noe configurations not realizable} and \eqref{eq: equivalent Condition02 No2 and Noe configurations not realizable} in $p_1$ if and only if (6) holds with $p < z/(2 + \sqrt{5})$. One implies   $p_1 > 0$.
By \eqref{eq: equivalent Condition01 No2 and Noe configurations not realizable}, it is implied that \eqref{eq: equivalent Condition03 No2 and Noe configurations not realizable} is equivalent to
\begin{equation}
\label{eq: equivalent02 Condition03 No2 and Noe configurations not realizable}
p_1 > \frac{p(3p^2 - 6zp + 2z^2)}{(2p - z)^2}.
\end{equation}
Based on \eqref{eq: equivalent Condition02 No2 and Noe configurations not realizable}, one implies that \eqref{eq: equivalent02 Condition03 No2 and Noe configurations not realizable} is equivalent to $(p^2+2zp-z^2)(p-z)^3(2p^3+10zp^2-7z^2p+z^3)>0$, which can be verified to be satisfied.

\textit{Sufficiency.}
Based on the discussion in the necessity part, there exists  $p_1 > 0$ such that \eqref{eq: equivalent Condition01 No2 and Noe configurations not realizable}--\eqref{eq: equivalent Condition03 No2 and Noe configurations not realizable} hold.
Let $m$, $\alpha$, $\beta$, and $\gamma$ satisfy \eqref{eq: 101 No2 and Noe configurations} and \eqref{eq: alpha No2 and Noe configurations}--\eqref{eq: gamma No2 and Noe configurations}, which obviously implies that $\alpha$, $\beta$, $\gamma$, $m$ $> 0$.
Therefore, \eqref{eq: Condition01 Noe configurations not realizable}--\eqref{eq: Condition03 Noe configurations not realizable}
and \eqref{eq: 102 No2 and Noe configurations}--\eqref{eq: 104 No2 and Noe configurations} hold. Therefore, $Z(s)$ can be written in the form of \eqref{eq: Z1 + Z2 No2 and Noe configurations}. Decompose
$Z(s)$ as $Z(s) = Z_1(s) + Z_2(s)$, where $Z_1(s)$ is in the form of \eqref{eq: Z1 No2 (s+p1)(s+p)} where $m$, $p_1$, $p$ $> 0$ and $Z_2(s)$ is in the form of \eqref{eq: 02 Noe configurations not realizable} where $\alpha$, $\beta$ $> 0$.
By letting $R_1 = m$, $L_1 = m/(p+p_1)$, and $C_1 = (p+p_1)/(mpp_1)$,
$Z_1(s)$ is realizable as in Fig.~\ref{fig: Two-reactive configurations N1}(b).
Let $C_{21}$, $C_{22}$, $R_{21}$, and $L_{21}$ satisfy \eqref{eq: C21 Noe configurations not realizable}--\eqref{eq: L21 Noe configurations not realizable}. Since \eqref{eq: Condition03 Noe configurations not realizable} holds, $C_{21}$, $C_{22}$, $R_{21}$, $L_{21}$ $> 0$. Based on the discussion in the proof of Lemma~11, it follows that \eqref{eq: 03 Noe configurations not realizable}--\eqref{eq: 08 Noe configurations not realizable} hold because \eqref{eq: Condition01 Noe configurations not realizable} and \eqref{eq: Condition02 Noe configurations not realizable} are satisfied.
Therefore, $Z_2(s)$ can be realized as in Fig.~\ref{fig: Three-reactive Four-element N2}(e).
The sufficiency part is proved.
\end{proof}

%%%%%%%%%%%%%%%%%%%%%%%%%%%%%%%%%%%%%%%%

\begin{lemma}
\label{lemma: five-reactive-element configurations realizable 03}
A biquadratic impedance $Z(s) \in \mathcal{Z}_{p2,z2}$ not satisfying the condition of Lemma~3 cannot be realized
as  in Fig.~2(a), where $N_1$ is the configuration in Fig.~\ref{fig: Two-reactive configurations N1}(c) and $N_2$ is one of the configurations in Fig.~\ref{fig: Three-reactive Four-element N2}(a) and \ref{fig: Three-reactive Four-element N2}(f).
\end{lemma}
\begin{proof}
It has been shown that the impedance of the configuration in Fig.~\ref{fig: Two-reactive configurations N1}(c) is in the form of $Z_1(s) = s(R_1L_1C_1 s + L_1)/(L_1C_1 s^2 + R_1C_1 s + 1)$, and the impedance of the configuration in Fig.~\ref{fig: Three-reactive Four-element N2}(a) is in the form of
\eqref{eq: 01 Noa configurations not realizable}.
Since it is assumed that the condition of Lemma~3
does not hold, $Z(s) \in \mathcal{Z}_{p2,z2}$ cannot be realized with fewer than five reactive elements.
If $Z(s)$ is realizable as  in Fig.~2(a), where $N_1$ is the configuration  in Fig.~\ref{fig: Two-reactive configurations N1}(c) and $N_2$ is the configuration in Fig.~\ref{fig: Three-reactive Four-element N2}(a), then the impedance of $N_1$ is in the form of
\begin{equation}  \label{eq: Z1 No3 (s+p1)(s+p)}
Z_1(s) = \frac{ms(s + q)}{(s+p_1)(s+p)},
\end{equation}
where $m$, $p_1$, $p$ $> 0$ and $q = p_1p/(p_1 + p)$, and the impedance of $N_2$ is in the form of \eqref{eq: 02 Noa configurations not realizable}, where $\alpha$, $\gamma$ $> 0$, and moreover \eqref{eq: Condition01 Noa configurations not realizable} and \eqref{eq: Condition02 Noa configurations not realizable} hold. Furthermore,
\begin{equation}  \label{eq: Z1 + Z2 No3 and Noa configurations}
\begin{split}
&Z(s) = Z_1(s) + Z_2(s) = \\
& \frac{ms^3 + \frac{mp^2 + (2mp_1+\alpha)p + \alpha p_1}{p_1 + p}s^2 + \frac{mp_1p^2}{p_1+p}s + \gamma}{(s+p_1)(s+p)^2}.
\end{split}
\end{equation}
Comparing (1) with \eqref{eq: Z1 + Z2 No3 and Noa configurations}, one obtains \begin{align}
m &= k,
\label{eq: 101 No3 and Noa configurations} \\
\frac{mp^2 + (2mp_1+\alpha)p + \alpha p_1}{p + p_1} &= k(p_1 + 2z),
\label{eq: 102 No3 and Noa configurations}  \\
\frac{mp_1p^2}{p + p_1} &= kz(2p_1+z),
\label{eq: 103 No3 and Noa configurations}  \\
\gamma &= kp_1z^2.
\label{eq: 104 No3 and Noa configurations}
\end{align}
Then, it follows from \eqref{eq: 101 No3 and Noa configurations} and \eqref{eq: 102 No3 and Noa configurations} that
\begin{equation}  \label{eq: alpha No3 and Noa configurations}
\alpha = \frac{k(p_1^2 - (p-2z)p_1 - p(p-2z))}{p+p_1}.
\end{equation}
Substituting \eqref{eq: 101 No3 and Noa configurations}, \eqref{eq: 104 No3 and Noa configurations}, and \eqref{eq: alpha No3 and Noa configurations} into \eqref{eq: Condition01 Noa configurations not realizable}, \eqref{eq: Condition02 Noa configurations not realizable}, and \eqref{eq: 103 No3 and Noa configurations} yields
\begin{equation} \label{eq: equivalent Condition01 No3 and Noa configurations not realizable}
\begin{split}
2pp_1^3 &- (p^2 - 4zp + z^2)p_1^2 \\
&- p(3p^2 - 6zp + z^2)p_1 - p^3(p-2z) > 0,
\end{split}
\end{equation}
\begin{equation} \label{eq: equivalent Condition02 No3 and Noa configurations not realizable}
\begin{split}
(p-z)(p+z)p_1^2 - p(p^2 &- 2zp + 3z^2)p_1 \\
&- p^2(p^2 - 2zp + 2z^2) = 0,
\end{split}
\end{equation}
\begin{equation} \label{eq: equivalent Condition03 No3 and Noa configurations not realizable}
\begin{split}
2zp_1^2 - (p^2 - 2zp - z^2)p_1 + z^2p  = 0,
\end{split}
\end{equation}
respectively. By \eqref{eq: alpha No3 and Noa configurations}, the assumption of $\alpha > 0$ implies
\begin{equation}  \label{eq: equivalent Condition04 No3 and Noa configurations not realizable}
p_1^2 - (p-2z)p_1 - p(p-2z) > 0.
\end{equation}
By calculation,
the resultant of \eqref{eq: equivalent Condition02 No3 and Noa configurations not realizable} and \eqref{eq: equivalent Condition03 No3 and Noa configurations not realizable} in $p_1$ is
$-p^4(p^6 - 8zp^5 + 20z^2p^4 - 28z^3p^3 + 21z^4p^2 - 12z^5p + 2z^6)$.
Since the condition of
Lemma~3 does not hold, there exists at least one common root in $p_1$  between \eqref{eq: equivalent Condition02 No3 and Noa configurations not realizable} and \eqref{eq: equivalent Condition03 No3 and Noa configurations not realizable} if and only if
\begin{equation}
\label{eq: Condition five-reactive-element configurations realizable 03}
p^6 - 8zp^5 + 20z^2p^4 -28z^3p^3 + 21z^4p^2 -12z^5p + 2z^6 = 0
\end{equation}
holds with $p > (2 + \sqrt{5})z$, which implies that the condition of Lemma~3 holds.
This further implies that  $p_1 > 0$.
From \eqref{eq: equivalent Condition03 No3 and Noa configurations not realizable}, it follows that \eqref{eq: equivalent Condition01 No3 and Noa configurations not realizable} is equivalent to
\begin{equation}
\label{eq: equivalent02 Condition01 No3 and Noa configurations not realizable}
p_1 > \frac{zp(2p-3z)}{(p-z)(p-3z)}.
\end{equation}
By \eqref{eq: equivalent Condition02 No3 and Noa configurations not realizable}, one has that \eqref{eq: equivalent02 Condition01 No3 and Noa configurations not realizable} is further equivalent to
$p^{10} - 12zp^9 + 54z^2p^8 - 114z^3p^7 + 100z^4p^6 + 40z^5p^5 - 164z^6p^4 + 142z^7p^3 - 65z^8p^2 + 16z^9p - 2z^{10} > 0$, which can indeed be verified to be true.

It has been shown that the impedance of the configuration in Fig.~\ref{fig: Two-reactive configurations N1}(c) is in the form of
\eqref{eq: Z1 c configuration}, and the impedance of the configuration in Fig.~8(f) is in the form of
\eqref{eq: 01 Nof configurations not realizable}.
If $Z(s)$ is realizable as in Fig.~2(a), where $N_1$ is the configuration in Fig.~\ref{fig: Two-reactive configurations N1}(c) and $N_2$ is the configuration in Fig.~\ref{fig: Three-reactive Four-element N2}(f), then the impedance of $N_1$ is in the form of \eqref{eq: Z1 No3 (s+p1)(s+p)}, where $m$, $p_1$, $p$ $> 0$ and $q = p_1p/(p_1 + p)$, and the impedance of $N_2$ is in the form of \eqref{eq: 02 Noe configurations not realizable}, where $\alpha$, $\beta$, $\gamma$ $> 0$ and \eqref{eq: Condition01 Nof configurations not realizable}--\eqref{eq: Condition03 Nof configurations not realizable} hold. Furthermore, one obtains \eqref{eq: Z1 + Z2 No3 and Noe configurations}--\eqref{eq: 104 No3 and Noe configurations}. It follows from \eqref{eq: Condition03 Nof configurations not realizable} and \eqref{eq: 104 No3 and Noe configurations} that
\begin{equation}  \label{eq: alpha No3 and Nof configurations}
\alpha = \frac{kp_1z^2(p_1+2p)}{2p(p+p_1)^2}.
\end{equation}
Then, \eqref{eq: Condition01 Nof configurations not realizable} and \eqref{eq: alpha No3 and Nof configurations} yield
\begin{equation}  \label{eq: beta No3 and Nof configurations}
\beta = \frac{kp_1z^2(p_1 + 2p)^2}{2p(p+p_1)^2}.
\end{equation}
Substituting \eqref{eq: 101 No3 and Noe configurations}, \eqref{eq: 104 No3 and Noe configurations}, \eqref{eq: alpha No3 and Noe configurations}, and \eqref{eq: beta No3 and Nof configurations} into \eqref{eq: Condition02 Nof configurations not realizable}, \eqref{eq: 102 No3 and Noe configurations}, and \eqref{eq: 103 No3 and Noe configurations} gives
\begin{align}
\frac{kp_1^2z^2p}{2(p+p_1)^2} &> 0,
\label{eq: equivalent Condition01 No3 and Nof configurations not realizable}  \\
2pp_1^3 + z(4p - z)p_1^2 - 2p(2p^2 - 4zp + z^2)p_1 - 2p^3(p-2z) &= 0,
\label{eq: equivalent Condition02 No3 and Nof configurations not realizable}   \\
z(4p-z)p_1^3 - 2p(p^2 - 4zp + z^2)p_1^2 - 2p^3(p-2z)p_1 + 2z^2p^3 &= 0,
\label{eq: equivalent Condition03 No3 and Nof configurations not realizable}
\end{align}
respectively. It is obvious that \eqref{eq: equivalent Condition01 No3 and Nof configurations not realizable} holds.
It is calculated that the resultant of \eqref{eq: equivalent Condition02 No3 and Nof configurations not realizable} and \eqref{eq: equivalent Condition03 No3 and Nof configurations not realizable} in
$p_1$ is $-4p^6z^4(2p^4 - 12zp^3 + 18z^2p^2 - 8z^3p + z^4)^2$.
Since the condition of
Lemma~3 does not hold,
it is implied that there exists at least one common root between
\eqref{eq: equivalent Condition02 No3 and Nof configurations not realizable} and \eqref{eq: equivalent Condition03 No3 and Nof configurations not realizable} if and only if
\begin{equation}  \label{eq: Condition five-reactive-element configurations realizable 05}
2p^4 - 12zp^3 + 18z^2p^2 - 8z^3p + z^4 = 0
\end{equation}
holds  with $p < z/(2 + \sqrt{5})$, which implies that the condition  of Lemma~3 holds.
This further implies that  $p_1 > 0$.
\end{proof}

\begin{lemma}
\label{lemma: five-reactive-element configurations realizable 04}
A biquadratic impedance $Z(s) \in \mathcal{Z}_{p2,z2}$ not satisfying the condition of Lemma~3 is realizable as in Fig.~2(a), where $N_1$ is the configuration in Fig.~\ref{fig: Two-reactive configurations N1}(c) and $N_2$ is the configuration in Fig.~\ref{fig: Three-reactive Four-element N2}(e) (that is, the configuration in Fig.~5(a)   whose one-terminal-pair labeled graph is $N_{5a}$), if and only if
\begin{equation}  \label{eq: Condition five-reactive-element configurations realizable 04}
\begin{split}
p^{10} &- 16zp^9 + 118z^2p^8 - 476z^3p^7
\\
&+ 1066z^4p^6 - 1372z^5p^5
+ 1064z^6p^4   \\
&- 524z^7p^3+ 161z^8p^2 - 28z^9p + 2z^{10} = 0
\end{split}
\end{equation}
and  $p < z/(2 + \sqrt{5})$ (it can be verified that the equation $\eta^{10} - 16 \eta^9 + 118 \eta^8  - 476 \eta^7
+ 1066 \eta^6 - 1372 \eta^5 + 1064 \eta^4 - 524 \eta^3  + 161 \eta^2 - 28 \eta + 2 = 0$ has only one distinct root for $\eta \in (0, 1/(2 + \sqrt{5}))$).
\end{lemma}
\begin{proof}
\textit{Necessity.}
It has been shown that the impedance of the configuration in Fig.~\ref{fig: Two-reactive configurations N1}(c) is in the form of
\eqref{eq: Z1 c configuration}, and the impedance of the configuration in Fig.~\ref{fig: Three-reactive Four-element N2}(e) is in the form of
\eqref{eq: 01 Noe configurations not realizable}.
Since it is assumed that the condition of Lemma~3
does not hold, $Z(s) \in \mathcal{Z}_{p2,z2}$ cannot be realized with fewer than five reactive elements.
If $Z(s)$ is realizable as  in Fig.~2(a), where $N_1$ is the configuration in Fig.~\ref{fig: Two-reactive configurations N1}(c) and $N_2$ is the configuration  in Fig.~\ref{fig: Three-reactive Four-element N2}(e), then the impedance of $N_1$ is in the form of \eqref{eq: Z1 No3 (s+p1)(s+p)},
where $m$, $p_1$, $p$ $> 0$ and $q = p_1p/(p_1 + p)$, and the impedance of $N_2$ is in the form of \eqref{eq: 02 Noe configurations not realizable}, where $\alpha$, $\beta$, $\gamma$ $> 0$ and
\eqref{eq: Condition01 Noe configurations not realizable}--\eqref{eq: Condition03 Noe configurations not realizable} hold.
Furthermore, one obtains
\begin{equation}  \label{eq: Z1 + Z2 No3 and Noe configurations}
Z(s) = Z_1(s) + Z_2(s) = \frac{ms^3 + \frac{mp^2 + (2mp_1+\alpha)p + \alpha p_1}{p_1 + p}s^2 + \frac{mp_1p^2 + \beta p + \beta p_1}{p_1+p}s + \gamma}{(s+p_1)(s+p)^2}.
\end{equation}
Combining (1) with \eqref{eq: Z1 + Z2 No3 and Noe configurations}, one obtains
\begin{align}
m &= k,
\label{eq: 101 No3 and Noe configurations} \\
\frac{mp^2 + (2mp_1+\alpha)p + \alpha p_1}{p + p_1} &= k(p_1 + 2z),
\label{eq: 102 No3 and Noe configurations}  \\
\frac{mp_1p^2 + \beta p + \beta p_1}{p + p_1} &= kz(2p_1+z),
\label{eq: 103 No3 and Noe configurations}  \\
\gamma &= kp_1z^2.
\label{eq: 104 No3 and Noe configurations}
\end{align}
It follows from \eqref{eq: Condition01 Noe configurations not realizable} and \eqref{eq: 104 No3 and Noe configurations} that
\begin{equation}  \label{eq: alpha No3 and Noe configurations}
\alpha = \frac{kp_1z^2}{p(p+2p_1)}.
\end{equation}
By \eqref{eq: Condition02 Noe configurations not realizable}, \eqref{eq: 104 No3 and Noe configurations}, and \eqref{eq: alpha No3 and Noe configurations}, one obtains
\begin{equation}  \label{eq: beta No3 and Noe configurations}
\beta = \frac{2kp_1z^2(p+p_1)^2}{(p+2p_1)^2p}.
\end{equation}
Substituting \eqref{eq: 101 No3 and Noe configurations} and \eqref{eq: 104 No3 and Noe configurations}--\eqref{eq: beta No3 and Noe configurations} into \eqref{eq: Condition03 Noe configurations not realizable}, \eqref{eq: 102 No3 and Noe configurations}, and \eqref{eq: 103 No3 and Noe configurations} gives
\begin{align}
\frac{kp_1^2z^2}{(p+2p_1)^2} &> 0,
\label{eq: equivalent Condition01 No3 and Noe configurations not realizable} \\
2pp_1^3 - (p^2-4zp+z^2)p_1^2 - p(3p^2-6zp+z^2)p_1 - p^3(p-2z) &= 0,  \label{eq: equivalent Condition02 No3 and Noe configurations not realizable}
\end{align}
and
\begin{equation} \label{eq: equivalent Condition03 No3 and Noe configurations not realizable}
\begin{split}
2z(4p-z)p_1^4 -2p(2p^2-8zp+z^2)p_1^3 &- 2p^2(2p^2-5zp-z^2)p_1^2 \\
&- p^3(p+z)(p-3z)p_1 + z^2p^4 = 0,
\end{split}
\end{equation}
respectively. It is obvious that \eqref{eq: equivalent Condition01 No3 and Noe configurations not realizable} holds. It is calculated that the resultant of \eqref{eq: equivalent Condition02 No3 and Noe configurations not realizable} and \eqref{eq: equivalent Condition03 No3 and Noe configurations not realizable} in $p_1$ is
$-4z^3p^{10}(4p-z)(p^{10} - 16zp^9 + 118z^2p^8 - 476z^3p^7 + 1066z^4p^6 - 1372z^5p^5 + 1064z^6p^4 - 524z^7p^3 + 161z^8p^2 - 28z^9p + 2z^{10})$. Since the condition of
Lemma~3 does not hold, it is implied that there exists at least one common root between \eqref{eq: equivalent Condition02 No3 and Noe configurations not realizable} and \eqref{eq: equivalent Condition03 No3 and Noe configurations not realizable} if and only if \eqref{eq: Condition five-reactive-element configurations realizable 04} holds  with $p < z/(2 + \sqrt{5})$.
This further implies that  $p_1 > 0$.

\textit{Sufficiency.}
Based on the discussion in the necessity part, there exists $p_1 > 0$ such that \eqref{eq: equivalent Condition01 No3 and Noe configurations not realizable}--\eqref{eq: equivalent Condition03 No3 and Noe configurations not realizable} hold. Let $m$, $\gamma$, $\alpha$, and $\beta$ satisfy \eqref{eq: 101 No3 and Noe configurations} and \eqref{eq: 104 No3 and Noe configurations}--\eqref{eq: beta No3 and Noe configurations}, which implies $\alpha$, $\beta$, $\gamma$, $m$ $> 0$. Therefore, \eqref{eq: Condition01 Noe configurations not realizable}--\eqref{eq: Condition03 Noe configurations not realizable},
\eqref{eq: 102 No3 and Noe configurations}, and \eqref{eq: 103 No3 and Noe configurations} hold. Therefore,
$Z(s)$ can be written in the form of \eqref{eq: Z1 + Z2 No3 and Noe configurations}. Decompose $Z(s)$ as $Z(s) = Z_1(s) + Z_2(s)$, where $Z_1(s)$ is in the form of \eqref{eq: Z1 No3 (s+p1)(s+p)} with $m$, $p_1$, $p$ $> 0$ and $q = p_1p/(p_1 + p)$, and $Z_2(s)$ is in the form of \eqref{eq: 02 Noe configurations not realizable} with $\alpha$, $\beta$, $\gamma$ $> 0$.
By letting $R_1 = m$, $C_1 = 1/(mq)$, and $L_1 = m/(p_1 + p)$, $Z_1(s)$ is realizable as in Fig.~\ref{fig: Two-reactive configurations N1}(c).
Let $C_{21}$, $C_{22}$, $R_{21}$, and $L_{21}$ satisfy \eqref{eq: C21 Noe configurations not realizable}--\eqref{eq: L21 Noe configurations not realizable}. Since \eqref{eq: Condition03 Noe configurations not realizable} hold, $C_{21}$, $C_{22}$, $R_{21}$, $L_{21}$ $> 0$. Based on the discussion in the proof of Lemma~11,  \eqref{eq: 03 Noe configurations not realizable}--\eqref{eq: 08 Noe configurations not realizable} hold. Therefore, $Z_2(s)$ can be realized as the configuration in Fig.~\ref{fig: Three-reactive Four-element N2}(e). The sufficiency part is proved.
\end{proof}

%%%%%%%%%%%%%%%%%%%%%%%%%%%%%%%%%%%%%%%%%%%%%%%%%%%%%%%%%%%%%%%%%%%%%%%%%

%\section*{Acknowledgment}
%The authors would like to thank...

% Can use something like this to put references on a page
% by themselves when using endfloat and the captionsoff option.
\ifCLASSOPTIONcaptionsoff
  \newpage
\fi

\end{document}